\documentclass[11pt,a4paper]{article}
\usepackage[utf8]{inputenc}
\usepackage[english]{babel}
\usepackage[T1]{fontenc}
\usepackage{amsmath,amsfonts,amssymb,amsthm}
\usepackage{mathtools}
\usepackage{lmodern}
\usepackage{graphicx}
\usepackage{setspace}
\usepackage[hidelinks]{hyperref}
\usepackage{multirow}
\usepackage{bbding}
\usepackage{pifont}
\usepackage{algorithm}
\usepackage{algorithmicx}
\usepackage{algpseudocode}
\algnewcommand\algorithmicinput{\textbf{Input:}}
\algnewcommand\Input{\item[\algorithmicinput]}
\algnewcommand\algorithmicoutput{\textbf{Output:}}
\algnewcommand\Output{\item[\algorithmicoutput]}
\usepackage{chngcntr}
\counterwithin{figure}{section}
\counterwithin{table}{section}
\counterwithin{algorithm}{section}
\counterwithin{equation}{section}

\usepackage[left=18mm,right=18mm,top=18mm,bottom=18mm]{geometry}
\author{Guillaume Olikier\footnotemark[2] \and Kyle A. Gallivan\footnotemark[3] \and P.-A. Absil\footnotemark[4]}
\title{Low-rank optimization methods based on projected projected-gradient descent that accumulate at Bouligand stationary points\footnotemark[1]}

\newcommand{\N}{\mathbb{N}}

\newcommand{\R}{\mathbb{R}}

\DeclareMathOperator{\dom}{dom}
\DeclareMathOperator{\im}{im}
\newcommand{\dist}{d}
\newcommand{\proj}[2]{P_{#1}(#2)}
\newcommand{\ball}{B}
\newcommand{\norm}[1]{\Vert #1 \Vert}
\newcommand{\ip}[2]{\left\langle #1, #2 \right\rangle}
\DeclareMathOperator*{\argmin}{argmin}
\DeclareMathOperator*{\lip}{Lip}

\DeclareMathOperator{\s}{s}

\DeclareMathOperator*{\outlim}{\overline{Lim}}
\newcommand{\setmapsto}{\multimap}

\makeatletter
\def\widebreve{\mathpalette\wide@breve}
\def\wide@breve#1#2{\sbox\z@{$#1#2$}%
     \mathop{\vbox{\m@th\ialign{##\crcr
\kern0.08em\brevefill#1{0.8\wd\z@}\crcr\noalign{\nointerlineskip}%
                    $\hss#1#2\hss$\crcr}}}\limits}
\def\brevefill#1#2{$\m@th\sbox\tw@{$#1($}%
  \hss\resizebox{#2}{\wd\tw@}{\rotatebox[origin=c]{90}{\upshape(}}\hss$}
\makeatletter

\newcommand{\gencone}[4]{{#1}_{#2}^{#4}(#3)} 
\newcommand{\tancone}[2]{\gencone{T}{#1}{#2}{}} 
\newcommand{\restancone}[2]{\gencone{\widebreve{T}}{#1}{#2}{}} 
\newcommand{\norcone}[2]{\gencone{N}{#1}{#2}{}} 
\newcommand{\regnorcone}[2]{\gencone{\widehat{N}}{#1}{#2}{}} 

\newcommand{\oshort}[1]{\mkern 0.9mu\overline{\mkern-0.9mu#1\mkern-0.9mu}\mkern 0.9mu}
\newcommand{\ushort}[1]{\mkern 0.9mu\underline{\mkern-0.9mu#1\mkern-0.9mu}\mkern 0.9mu}

\newcommand{\tp}{\top}

\DeclareMathOperator{\tr}{tr}
\DeclareMathOperator{\rank}{rank}
\newcommand{\st}{\mathrm{St}}
\DeclareMathOperator{\diag}{diag}

\newcommand{\pgd}{\mathrm{PGD}}
\newcommand{\ppgd}{\mathrm{P}^2\mathrm{GD}}
\newcommand{\ppgdr}{\mathrm{P}^2\mathrm{GDR}}
\newcommand{\rfd}{\mathrm{RFD}}
\newcommand{\rfdr}{\mathrm{RFDR}}
\newcommand{\hrtr}{\mathrm{HRTR}}

\newtheorem{theorem}{Theorem}[section]
\newtheorem{proposition}[theorem]{Proposition}
\newtheorem{lemma}[theorem]{Lemma}
\newtheorem{corollary}[theorem]{Corollary}
\theoremstyle{definition}
\newtheorem{definition}[theorem]{Definition}
\newtheorem{example}[theorem]{Example}
\newtheorem{remark}[theorem]{Remark}

\renewenvironment{thebibliography}[1]{
  \begin{oldthebibliography}{#1}
    \setlength{\itemsep}{7pt}
    \setlength{\parskip}{0pt}
}
{
  \end{oldthebibliography}
}

\begin{document}
\renewcommand{\thefootnote}{\fnsymbol{footnote}}
\footnotetext[1]{This work was supported by the Fonds de la Recherche Scientifique -- FNRS and the Fonds Wetenschappelijk Onderzoek -- Vlaanderen under EOS Project no 30468160, by the Fonds de la Recherche Scientifique -- FNRS under Grant no T.0001.23, by ERC grant 786854 G-Statistics from the European Research Council under the European Union's Horizon 2020 research and innovation program, and by the French government through the 3IA Côte d'Azur Investments ANR-23-IACL-0001 managed by the National Research Agency. K.~A. Gallivan was partially supported by the U.S. National Science Foundation under grant CIBR 1934157. This work was done in part while K.~A. Gallivan was visiting UCLouvain, supported by the ``Professeurs et chercheurs visiteurs'' budget of the Science and Technology Sector.}
\footnotetext[2]{Institute of Mathematics, \'{E}cole polytechnique fédérale de Lausanne (EPFL), 1015 Lausanne, Switzerland \mbox{(\href{mailto:guillaume.olikier@epfl.ch}{\nolinkurl{guillaume.olikier@epfl.ch}})}.}
\footnotetext[3]{Department of Mathematics, Florida State University, 1017 Academic Way, Tallahassee, FL 32306-4510, USA (\href{mailto:kgallivan@fsu.edu}{\nolinkurl{kgallivan@fsu.edu}}).}
\footnotetext[4]{ICTEAM Institute, UCLouvain, Avenue Georges Lemaître 4, 1348 Louvain-la-Neuve, Belgium \mbox{(\href{mailto:pa.absil@uclouvain.be}{\nolinkurl{pa.absil@uclouvain.be}})}.}
\renewcommand{\thefootnote}{\arabic{footnote}}

\maketitle

\begin{abstract}
This paper considers the problem of minimizing a differentiable function with locally Lipschitz continuous gradient on the algebraic variety of real matrices of upper-bounded rank. This problem is known to enable the formulation of various machine learning or signal processing tasks such as dimensionality reduction, collaborative filtering, and signal recovery. Several definitions of stationarity exist for this nonconvex problem. Among them, Bouligand stationarity is the strongest necessary condition for local optimality. This paper proposes two first-order methods that generate a sequence in the variety whose accumulation points are Bouligand stationary. The first method combines the well-known projected projected-gradient descent map with a rank reduction mechanism. The second method is a hybrid of projected gradient descent and projected projected-gradient descent. Both methods stand out in the field of low-rank optimization methods when considering their convergence properties, their streamlined design, their typical computational cost per iteration, and their empirically observed numerical performance. The theoretical framework used to analyze the proposed methods is of independent interest.
\medskip

\noindent
\textbf{Keywords:}
convergence analysis, stationary point, critical point, low-rank optimization, determinantal variety, steepest descent, tangent and normal cones, Clarke regularity, weighted low-rank approximation.
\medskip

\noindent
\textbf{Mathematics Subject Classification:} 14M12, 65K10, 90C26, 90C30, 40A05.
\end{abstract}

\section{Introduction}
\label{sec:Introduction}
Low-rank optimization concerns the problem of minimizing a real-valued function over a space of matrices subject to an upper bound on their rank. Applications abound in various fields of science and engineering. The present work addresses specifically the problem
\begin{equation}
\label{eq:OptiDeterminantalVariety}
\min_{X \in \R_{\le r}^{m \times n}} f(X)
\end{equation}
of minimizing a differentiable function $f : \R^{m \times n} \to \R$ with locally Lipschitz continuous gradient on the determinantal variety
\begin{equation}
\label{eq:RealDeterminantalVariety}
\R_{\le r}^{m \times n} \coloneq \{X \in \R^{m \times n} \mid \rank X \le r\},
\end{equation}
$m$, $n$, and $r$ being positive integers such that $r < \min\{m, n\}$.

Determinantal varieties constitute an important topic in algebraic geometry, especially over algebraically closed fields such as the field of complex numbers; see, e.g., \cite[Chapter~II]{ArbarelloCornalbaGriffithsHarris}, \cite[\S 1C]{BrunsVetter}, \cite[Lecture~9]{Harris}, \cite[Chapter~10]{LakshmibaiBrown}, and~\cite{Tsakiris2024} and references therein. Their geometry is well understood, which is particularly useful in the context of problem~\eqref{eq:OptiDeterminantalVariety}. For example, the smooth and singular parts of $\R_{\le r}^{m \times n}$ are known to be
\begin{equation*}
\R_r^{m \times n} \coloneq \{X \in \R^{m \times n} \mid \rank X = r\}
\end{equation*}
and
\begin{equation*}
\R_{< r}^{m \times n} \coloneq \{X \in \R^{m \times n} \mid \rank X < r\},
\end{equation*}
respectively, and explicit formulas for the tangent cone to $\R_{\le r}^{m \times n}$ are available \cite{OlikierMlinaricAbsilUschmajew}.

Problem~\eqref{eq:OptiDeterminantalVariety} includes as special cases many important problems such as weighted low-rank approximation \cite{SrebroJaakkola,GillisGlineur2011}, robust affine rank minimization \cite{JainMekaDhillon,GoldfarbMa,JainNetrapalliSanghavi,WeiCaiChanLeung,TongMaChi,CarlssonGerosaOlsson,DingDrusvyatskiyFazelHarchaoui,Zhang}, various formulations of matrix completion \cite{TannerWei2013,ZhouYangZhaoYu,FreundGrigasMazumder,YangFengSuykens,NguyenKimShim,DingChen,GaoAbsil}, multi-task learning \cite{AbernethyBachEvgeniouVert}, matrix sensing \cite{ChiLuChen,ZhuLiTangWakin}, robust PCA \cite{HaLiuBarber2020}, and logistic PCA \cite{ParkKyrillidisCaramanisSanghavi}.
These problems are known to enable the formulation of several machine learning or signal processing tasks such as dimensionality reduction, collaborative filtering, and signal recovery.

\subsection{Stationarity notions on the determinantal variety}
\label{subsec:StationarityNotionsDeterminantalVariety}
In general, problem~\eqref{eq:OptiDeterminantalVariety} is intractable \cite{GillisGlineur2011} and optimization methods are only expected to find a stationary point of this problem \cite{SchneiderUschmajew2015,UschmajewVandereycken2015,ZhouEtAl2016,HosseiniLukeUschmajew2019,LiSongXiu2019,HaLiuBarber2020,LiuBarber2020,UschmajewVandereycken2020IMA,JiaEtAl,LevinKileelBoumal2023,LiLuo2023,OlikierAbsil2023,OlikierUschmajewVandereycken,LevinKileelBoumal2025,Pauwels,OlikierWaldspurger,RebjockBoumal}.
A point $X \in \R_{\le r}^{m \times n}$ is said to be stationary for~\eqref{eq:OptiDeterminantalVariety} if $-\nabla f(X)$ is normal to $\R_{\le r}^{m \times n}$ at $X$. Several definitions of normality exist. Each one yields a definition of stationarity such that, possibly under mild regularity assumptions on~$f$, every local minimizer of $f|_{\R_{\le r}^{m \times n}}$ is stationary for~\eqref{eq:OptiDeterminantalVariety}. In particular, each of the three notions of normality in \cite[Definition~6.3 and Example~6.16]{RockafellarWets}, namely normality in the general sense, in the regular sense, and in the proximal sense, yields an important definition of stationarity. On $\R_{\le r}^{m \times n}$, normality in the proximal sense is equivalent to normality in the regular sense \cite[Proposition~7.1]{OlikierWaldspurger}. The sets of general and regular normals to $\R_{\le r}^{m \times n}$ at $X \in \R_{\le r}^{m \times n}$ are closed cones called the \emph{normal cone} to $\R_{\le r}^{m \times n}$ at $X$ and the \emph{regular normal cone} to $\R_{\le r}^{m \times n}$ at $X$ and denoted by $\norcone{\R_{\le r}^{m \times n}}{X}$ and $\regnorcone{\R_{\le r}^{m \times n}}{X}$, respectively. These normal cones are reviewed in Section~\ref{subsec:TangentNormalCones}. Importantly, they are nested as follows: for every $X \in \R_{\le r}^{m \times n}$,
\begin{equation}
\label{eq:NestedNormalCones}
\regnorcone{\R_{\le r}^{m \times n}}{X} \subseteq \norcone{\R_{\le r}^{m \times n}}{X},
\end{equation}
and $\R_{\le r}^{m \times n}$ is said to be \emph{Clarke regular} at $X$ if the inclusion is an equality \cite[Definition~6.4]{RockafellarWets}. As observed in \cite[Remark~4]{SchneiderUschmajew2015}, $\R_{\le r}^{m \times n}$ is Clarke regular at $X \in \R_{\le r}^{m \times n}$ if and only if $\rank X = r$.

The definitions of stationarity based on the normal cone and the regular normal cone are given in Definition~\ref{def:M/B-Stationarity}; comments on the terminology can be found, e.g., in \cite[\S 3]{OlikierWaldspurger}.

\begin{definition}
\label{def:M/B-Stationarity}
For problem~\eqref{eq:OptiDeterminantalVariety}, a point $X \in \R_{\le r}^{m \times n}$ is said to be:
\begin{itemize}
\item \emph{Mordukhovich stationary (M-stationary)} if $-\nabla f(X) \in \norcone{\R_{\le r}^{m \times n}}{X}$;
\item \emph{Bouligand stationary (B-stationary)} if $-\nabla f(X) \in \regnorcone{\R_{\le r}^{m \times n}}{X}$.
\end{itemize}
\end{definition}

These two notions of stationarity depend on $f$ only through its restriction $f|_{\R_{\le r}^{m \times n}}$ (see \cite{LevinKileelBoumal2023} for B-stationarity and Proposition~\ref{prop:MstationarityRestriction} for M-stationarity); this is a desirable property for a notion of stationarity. They are equivalent only on the smooth part $\R_r^{m \times n}$ of $\R_{\le r}^{m \times n}$. For problem~\eqref{eq:OptiDeterminantalVariety}, B-stationarity is the strongest necessary condition for local optimality, as shown by the following gradient characterization, which holds at every $X \in \R_{\le r}^{m \times n}$ \cite[Theorem~2.5 and Proposition~7.1]{OlikierWaldspurger}:
\begin{equation}
\label{eq:GradientCharacterizationRegularNormalsRealDeterminantalVariety}
\regnorcone{\R_{\le r}^{m \times n}}{X} = \left\{-\nabla h(X) ~\left|~\begin{array}{l} h : \R^{m \times n} \to \R \text{ is differentiable},\\ \nabla h \text{ is locally Lipschitz continuous},\\ X \text{ is a local minimizer of } h|_{\R_{\le r}^{m \times n}} \end{array}\right. \right\}.
\end{equation}
In comparison, M-stationarity is a weaker notion of stationarity, which is considered unsatisfactory in \cite[\S 4]{HosseiniLukeUschmajew2019}, \cite[\S 1]{LevinKileelBoumal2023}, and \cite[\S 2.1]{Pauwels}. This is illustrated in Example~\ref{example:MstationarityWithoutBstationarity} and detailed in Remark~\ref{rem:MstationarityDeterminantalVariety}.

\begin{example}[{M-stationarity without B-stationarity \cite[Example~1]{Pauwels}}]
\label{example:MstationarityWithoutBstationarity}
Consider the instance of~\eqref{eq:OptiDeterminantalVariety} where $m \coloneq n \coloneq 2$, $r \coloneq 1$, and $f : X \mapsto\frac{1}{2}\norm{X-\diag(1, 0)}^2$. As detailed in Section~\ref{subsec:TangentNormalCones}, $\{0_{2 \times 2}\} = \regnorcone{\R_{\le 1}^{2 \times 2}}{0_{2 \times 2}} \subsetneq \norcone{\R_{\le 1}^{2 \times 2}}{0_{2 \times 2}} = \R_{\le 1}^{2 \times 2}$. Hence, $\R_{\le 1}^{2 \times 2}$ is not Clarke regular at $0_{2 \times 2}$. Furthermore, $\nabla f(0_{2 \times 2}) = -\diag(1, 0)$. Thus, since $-\nabla f(0_{2 \times 2}) \in \norcone{\R_{\le 1}^{2 \times 2}}{0_{2 \times 2}} \setminus \regnorcone{\R_{\le 1}^{2 \times 2}}{0_{2 \times 2}}$, $0_{2 \times 2}$ is M-stationary but not B-stationary. Moreover, for all $\alpha \in (0, \infty)$, $0_{2 \times 2}-\alpha\nabla f(0_{2 \times 2}) = \diag(\alpha, 0) \in \R_{\le 1}^{2 \times 2}$. Therefore, moving from $0_{2 \times 2}$ in the direction of $-\nabla f(0_{2 \times 2})$ enables to decrease $f$ while staying in $\R_{\le 1}^{2 \times 2}$. Hence, $0_{2 \times 2}$ is not a local minimizer of $f|_{\R_{\le 1}^{2 \times 2}}$, but this cannot be seen based on M-stationarity.
\end{example}

Furthermore, as explained in \cite{LevinKileelBoumal2023}, distinguishing convergence to a B-stationary point from convergence to an M-stationary point is difficult in the sense that it cannot be done based on standard measures of B-stationarity, which are zero only at B-stationary points, such as
\begin{equation}
\label{eq:NormProjectionNegativeGradientOntoTangentCone}
\R_{\le r}^{m \times n} \to \R : X \mapsto \norm{\proj{\tancone{\R_{\le r}^{m \times n}}{X}}{-\nabla f(X)}},
\end{equation}
which returns the norm of any projection of $-\nabla f(X)$ onto the tangent cone to $\R_{\le r}^{m \times n}$ at $X$ (see \eqref{eq:StationarityMeasureExplicitFormula}), and
\begin{equation}
\label{eq:DistanceNegativeGradientToRegularNormalCone}
\R_{\le r}^{m \times n} \to \R : X \mapsto \dist(-\nabla f(X), \regnorcone{\R_{\le r}^{m \times n}}{X}),
\end{equation}
which returns the distance from $-\nabla f(X)$ to $\regnorcone{\R_{\le r}^{m \times n}}{X}$.
Indeed, these stationarity measures can fail to be lower semicontinuous at a point where $\R_{\le r}^{m \times n}$ is not Clarke regular (Proposition~\ref{prop:NotClarkeRegularImpliesApocalyptic}). As a matter of fact, when applied to certain instances of problem~\eqref{eq:OptiDeterminantalVariety}, optimization methods such as $\ppgd$ and $\rfd$, reviewed in Section~\ref{subsec:StateOfTheArt}, generate a convergent sequence in the smooth part $\R_r^{m \times n}$ of $\R_{\le r}^{m \times n}$ along which the measures of B-stationarity defined in~\eqref{eq:NormProjectionNegativeGradientOntoTangentCone} and \eqref{eq:DistanceNegativeGradientToRegularNormalCone} converge to zero and whose limit, which is in the singular part $\R_{< r}^{m \times n}$ of $\R_{\le r}^{m \times n}$, is M-stationary but not B-stationary, a phenomenon formalized by the notion of \emph{apocalypse} in \cite{LevinKileelBoumal2023}. In contrast, the measure of M-stationarity
\begin{equation}
\label{eq:DistanceNegativeGradientToNormalCone}
\R_{\le r}^{m \times n} \to \R : X \mapsto \dist(-\nabla f(X), \norcone{\R_{\le r}^{m \times n}}{X})
\end{equation}
is lower semicontinuous (Proposition~\ref{prop:LowerSemicontinuityMeasureMstationarity}). Therefore, if it goes to zero along a convergent sequence in $\R_{\le r}^{m \times n}$, then it is zero at the limit.

The coexistence of several notions of stationarity for an optimization problem is common. It arises notably in nonsmooth optimization, where several notions of subdifferential yield different notions of stationarity \cite{LiSoMa2020,PangRazaviyaynAlvarado,CuiPang}, and in mathematical programs with complementarity constraints, for which several notions of stationarity are presented and compared in \cite{ScheelScholtes}.

\subsection{State of the art}
\label{subsec:StateOfTheArt}
Five state-of-the-art methods aiming at solving problem~\eqref{eq:OptiDeterminantalVariety} are listed in Table~\ref{tab:OptimizationMethodsDeterminantalVariety} and described one by one in this section. Among them, $\pgd$, $\rfdr$, and \cite[Algorithm~1]{LevinKileelBoumal2023}---which we term Hooked Riemannian Trust-Region ($\hrtr$)---are the only three methods in the literature guaranteed to accumulate at B-stationary points of~\eqref{eq:OptiDeterminantalVariety}, which means that the accumulation points of every generated sequence are B-stationary for~\eqref{eq:OptiDeterminantalVariety}; the property of generating a sequence whose accumulation points are stationary (in some sense) is called ``(global) subsequential convergence'' to stationary points in \cite[Theorem~2(i)]{LiPong}, \cite[Theorem~1]{WangYinZeng}, \cite[Theorem~1]{FacchineiKungurtsevLamparielloScutari}, \cite[Theorem~7.2.4]{CuiPang}, \cite[Theorem~4.5]{BotDaoLi}, and \cite[Lemma~4.4]{CohenHallakTeboulle}. The other two state-of-the-art methods, namely $\ppgd$ and $\rfd$, do not always accumulate at B-stationary points of~\eqref{eq:OptiDeterminantalVariety}; see respectively \cite[\S 2.2]{LevinKileelBoumal2023} and \cite[\S 8]{OlikierAbsil2023}. However, $\rfd$ was used in~\cite{OlikierAbsil2023} to design a first-order method, $\rfdr$, provably accumulating at B-stationary points of~\eqref{eq:OptiDeterminantalVariety}. As explained in Section~\ref{subsec:Contribution}, in this paper, it is $\ppgd$ that is extended to produce two first-order methods, the proposed $\ppgdr$ and $\ppgd$--$\pgd$, mentioned in Table~\ref{tab:OptimizationMethodsDeterminantalVariety}, that provably accumulate at B-stationary points of~\eqref{eq:OptiDeterminantalVariety}.

\begin{table}[H]
\begin{center}
\begin{tabular}{*{4}{l}}
\hline
\emph{Method} & \emph{Order} & $\s(\cdot; f, \R_{\le r}^{m \times n}) \to 0$ & B-stationary\\[1pt]
\hline
$\pgd$ \cite[Algorithm~4.2]{OlikierWaldspurger} & $1$ & \ding{51} & \ding{51} \cite[Theorem~1.2]{OlikierWaldspurger}\\[1pt]
\hline
$\ppgd$ \cite[Algorithm~3]{SchneiderUschmajew2015} & $1$ & ? & \ding{55} \cite[\S 2.2]{LevinKileelBoumal2023}\\[1pt]
\hline
$\ppgdr$ (Definition~\ref{def:P2GDR}) & $1$ & \ding{51} & \ding{51} (Theorem~\ref{thm:P2GDRPolak})\\[1pt]
\hline
$\ppgd$--$\pgd$ (Definition~\ref{def:P2GD--PGD}) & $1$ & \ding{51} & \ding{51} (Theorem~\ref{thm:P2GD--PGD_Bstationary})\\[1pt]
\hline
$\rfd$ \cite[Algorithm~4]{SchneiderUschmajew2015} & $1$ & \ding{51} \cite[Theorem~3.10]{SchneiderUschmajew2015} & \ding{55} \cite[\S 8]{OlikierAbsil2023}\\[1pt]
\hline
$\rfdr$ \cite[Algorithm~3]{OlikierAbsil2023} & $1$ & \ding{51} & \ding{51} \cite[Theorem~6.2]{OlikierAbsil2023}\\[1pt]
\hline
$\hrtr$ \cite[Algorithm~1]{LevinKileelBoumal2023} & $2$ & \ding{51} & \ding{51} \cite[Theorem~1.1]{LevinKileelBoumal2023}\\[1pt]
\hline
\end{tabular}
\end{center}
\caption{B-stationarity properties of $\ppgdr$, $\ppgd$--$\pgd$, and five state-of-the-art methods aiming at solving problem~\eqref{eq:OptiDeterminantalVariety}. The property ``$\s(\cdot; f, \R_{\le r}^{m \times n}) \to 0$'' means that the measure of B-stationarity defined in~\eqref{eq:NormProjectionNegativeGradientOntoTangentCone} goes to zero along every convergent subsequence of the generated sequence. The property ``B-stationary'' means that the method accumulates at B-stationary points of~\eqref{eq:OptiDeterminantalVariety}, which is stronger (Proposition~\ref{prop:DeterminantalVarietyNoSerendipitousPoint}). The symbols ``\ding{51}'', ``\ding{55}'', and ``?'' respectively mean ``yes'', ``no'', and ``open question''.}
\label{tab:OptimizationMethodsDeterminantalVariety}
\end{table}

$\pgd$, which is short for \emph{projected gradient descent}, is appealing for its simplicity: given $X \in \R_{\le r}^{m \times n}$ as input, the $\pgd$ map---``map'' is added to the name of a method to refer to its iteration map---performs a backtracking projected line search along the direction of $-\nabla f(X)$, i.e., computes a projection $Y$ of $X-\alpha\nabla f(X)$ onto $\R_{\le r}^{m \times n}$ for decreasing values of the step size $\alpha \in (0, \infty)$ until $Y$ satisfies an Armijo condition. In the simplest version of $\pgd$, called \emph{monotone}, the Armijo condition ensures that the value of $f$ at the next iterate is smaller by a specified amount than the value at the current iterate. Alternatively, the value at the current iterate can be replaced with the maximum value of $f$ over a prefixed number of the previous iterates or with a weighted average of the values of $f$ at the previous iterates \cite{JiaEtAl,KanzowMehlitz,DeMarchi}. This version of $\pgd$ is called \emph{nonmonotone}.
This description of $\pgd$ brings to light its main drawback, already identified in \cite[\S 2.1]{WeiCaiChanLeung} and \cite[\S 9.4.4]{UschmajewVandereycken2020}, namely its computational cost per iteration, which can be prohibitive if $r \ll \min\{m, n\}$: if the gradient does not have a special structure such as sparse or low rank, which is the case, e.g., in weighted low-rank approximation \cite{GillisGlineur2011}, then every iteration requires projecting a possibly full-rank matrix onto $\R_{\le r}^{m \times n}$, which amounts to computing a truncated singular value decomposition (SVD) of this matrix. Furthermore, as pointed out in \cite{LevinKileelBoumal2023}, the iterates generated by $\pgd$ depend on the values taken by $f$ outside $\R_{\le r}^{m \times n}$, which are irrelevant to problem~\eqref{eq:OptiDeterminantalVariety}.

$\ppgd$, which is short for \emph{projected projected-gradient descent}, works like $\pgd$ except that it involves an additional projection: given $X \in \R_{\le r}^{m \times n}$ as input, the $\ppgd$ map performs a backtracking projected line search along a projection $G$ of $-\nabla f(X)$ onto $\tancone{\R_{\le r}^{m \times n}}{X}$, i.e., computes a projection $Y$ of $X+\alpha G$ onto $\R_{\le r}^{m \times n}$ for decreasing values of $\alpha \in (0, \infty)$ until $Y$ satisfies an Armijo condition. The structure of $\tancone{\R_{\le r}^{m \times n}}{X}$ implies that $X+\alpha G \in \R_{\le 2r}^{m \times n}$ for all $\alpha \in (0, \infty)$ and makes the projection onto $\R_{\le r}^{m \times n}$ easier to compute; see \cite[Algorithm~6]{Vandereycken2013}, \cite[\S 3.4]{SchneiderUschmajew2015}, \cite[\S 2.3]{WeiCaiChanLeung}, or \cite[\S 9.2.4]{UschmajewVandereycken2020}. Moreover, in the typical case where $\rank X = r$, projecting $-\nabla f(X)$ onto $\tancone{\R_{\le r}^{m \times n}}{X}$ merely involves orthonormal bases of $\im X$ and $\im X^\tp$, which are available since the projection onto $\R_{\le r}^{m \times n}$ is obtained by a truncated SVD, and matrix multiplications that each require at most $2mnr$ additions or multiplications of real numbers; see \cite[Algorithm~2]{SchneiderUschmajew2015}. Furthermore, as pointed out in \cite{LevinKileelBoumal2023}, the iterates generated by $\ppgd$ depend on $f$ only through its restriction $f|_{\R_{\le r}^{m \times n}}$.
In brief, the additional projection reduces the typical computational cost per iteration, but destroys the guarantee of accumulation at B-stationary points of~\eqref{eq:OptiDeterminantalVariety}.

$\rfd$, which is short for \emph{retraction-free descent}, relies on the \emph{restricted tangent cone} to $\R_{\le r}^{m \times n}$ (see \cite[\S 3]{OlikierAbsil2023} and \cite[\S 3.4]{SchneiderUschmajew2015}), a map that associates with every $X \in \R_{\le r}^{m \times n}$ a closed cone $ \restancone{\R_{\le r}^{m \times n}}{X} \subseteq \tancone{\R_{\le r}^{m \times n}}{X}$ such that $X+\restancone{\R_{\le r}^{m \times n}}{X} \subseteq \R_{\le r}^{m \times n}$ and, for all $Z \in \R^{m \times n}$, $\norm{\proj{\restancone{\R_{\le r}^{m \times n}}{X}}{Z}} \ge \frac{1}{\sqrt{2}} \norm{\proj{\tancone{\R_{\le r}^{m \times n}}{X}}{Z}}$.
Given $X \in \R_{\le r}^{m \times n}$ as input, the $\rfd$ map performs a backtracking line search along a projection $G$ of $-\nabla f(X)$ onto $\restancone{\R_{\le r}^{m \times n}}{X}$, which requires no retraction on $\R_{\le r}^{m \times n}$ (in the sense of \cite[\S 3.1.2]{HosseiniUschmajew2019}) since $X+\alpha G \in \R_{\le r}^{m \times n}$ for all $\alpha \in (0, \infty)$; the $\rfd$ map is thus retraction-free, i.e., it performs updates along straight lines. Moreover, the computational cost of projecting onto $\restancone{\R_{\le r}^{m \times n}}{X}$ is similar to that of projecting onto $\tancone{\R_{\le r}^{m \times n}}{X}$. However, as pointed out in \cite[\S 3.4]{SchneiderUschmajew2015}, saving the cost of computing a retraction does not necessarily confer to $\rfd$ a significant advantage over $\ppgd$ if $r \ll \min\{m, n\}$ since the latter requires projecting a point in $\R_{\le 2r}^{m \times n}$ onto $\R_{\le r}^{m \times n}$, which is typically much less expensive than evaluating $f$ or $\nabla f$. Furthermore, the directions in $\proj{\tancone{\R_{\le r}^{m \times n}}{X}}{-\nabla f(X)}$ are closer to $-\nabla f(X)$ than those in $\proj{\restancone{\R_{\le r}^{m \times n}}{X}}{-\nabla f(X)}$; while this does not imply that $\ppgd$ converges faster than $\rfd$, such an observation was made experimentally in \cite[\S 3.4]{SchneiderUschmajew2015}.

$\rfdr$ is $\rfd$ equipped with a \emph{rank reduction mechanism}, hence the ``R'' in the name. Given $X \in \R_{\le r}^{m \times n}$ as input, the $\rfdr$ map proceeds as follows: (i) it applies the $\rfd$ map to $X$, thereby producing a point $\tilde{X}$, (ii) if the $r$th singular value of $X$ is positive but smaller than or equal to some threshold $\Delta \in (0, \infty)$, it applies the $\rfd$ map to a projection $\hat{X}$ of $X$ onto $\R_{r-1}^{m \times n}$, then producing a point $\tilde{X}^\mathrm{R}$, and (iii) it outputs a point among $\tilde{X}$ and $\tilde{X}^\mathrm{R}$ that maximally decreases $f$. Thus, if the $r$th singular value of its input is not in $(0, \Delta]$, then the $\rfdr$ map reduces to the $\rfd$ map.

$\hrtr$ is a second-order method that deals with problem~\eqref{eq:OptiDeterminantalVariety} by \emph{lifting} it to a Riemannian manifold $\mathcal{M}$: given a smooth map $\varphi : \mathcal{M} \to \R^{m \times n}$ such that $\varphi(\mathcal{M}) = \R_{\le r}^{m \times n}$, it minimizes $f \circ \varphi$ on~$\mathcal{M}$; $\varphi$ is called a \emph{smooth lift} of $\R_{\le r}^{m \times n}$. This method takes advantage of the increasingly understood relations between the desirable points of the original problem and those of the lifted problem \cite{HaLiuBarber2020,LevinKileelBoumal2023,LevinKileelBoumal2025,RebjockBoumal}. This approach is appealing for its generality and elegance: it is applicable to every feasible set for which a suitable smooth lift is available and it generates iterates that depend on the cost function only through its restriction to the feasible set. However, on $\R_{\le r}^{m \times n}$, it suffers from a prohibitive computational cost in the frequently encountered case where $r \ll \min\{m, n\}$. Indeed, as explained in Section~\ref{subsec:ComparisonComputationalCostPerIteration}, every iteration of $\hrtr$, used with the rank factorization lift $\varphi : \R^{m \times r} \times \R^{n \times r} \to \R^{m \times n} : (L, R) \mapsto LR^\tp$, requires computing the smallest eigenvalue and an associated eigenvector of an order-$(m+n)r$ matrix representing the Hessian of the lifted cost function $f \circ \varphi$.

\subsection{Contribution}
\label{subsec:Contribution}
Motivated by the fact that $\ppgd$ performs better than $\rfd$ in the numerical experiment reported in \cite[\S 3.4]{SchneiderUschmajew2015} and that, in contrast with $\rfd$, $\ppgd$ does not require a restricted tangent cone, this paper extends $\ppgd$ to propose two first-order methods, called $\ppgdr$ and $\ppgd$--$\pgd$, that each generate a sequence in $\R_{\le r}^{m \times n}$ whose accumulation points are B-stationary for~\eqref{eq:OptiDeterminantalVariety} in the sense of Definition~\ref{def:M/B-Stationarity}. Both are based on the \emph{$\Delta$-rank} of a matrix: for every $\Delta \in (0, \infty)$, the $\Delta$-rank of $X \in \R^{m \times n}$, denoted by $\rank_\Delta X$, is defined as the number of singular values of $X$ that are greater than $\Delta$ \cite[(5.4.5)]{GolubVanLoan}.

$\ppgdr$ is $\ppgd$ equipped with a rank reduction mechanism, hence the ``R'' in the name. Given $X \in \R_{\le r}^{m \times n}$ as input, the $\ppgdr$ map (Algorithm~\ref{algo:P2GDRmap}) proceeds as follows: (i) for every $i \in \{0, \dots, \rank X - \rank_\Delta X\}$, it applies the $\ppgd$ map (Algorithm~\ref{algo:P2GDmap}) to a projection $\hat{X}^i$ of $X$ onto $\R_{\rank X - i}^{m \times n}$, thereby producing a point $\tilde{X}^i$, and (ii) it outputs a point among $\tilde{X}^0, \dots, \tilde{X}^{\rank X - \rank_\Delta X}$ that maximally decreases $f$.

As suggested by its name, $\ppgd$--$\pgd$ is a hybrid of $\ppgd$ and $\pgd$, which emerges naturally from the theoretical framework described in Section~\ref{sec:SufficientDescentMaps}. Given $X \in \R_{\le r}^{m \times n}$ as input, the $\ppgd$--$\pgd$ map (Algorithm~\ref{algo:P2GD--PGDmap}) applies the $\ppgd$ map if $\rank X = \rank_\Delta X$, and the monotone $\pgd$ map (Algorithm~\ref{algo:monotonePGDmap}) otherwise.

If the rank and $\Delta$-rank of the input are equal, then the $\ppgdr$ and $\ppgd$--$\pgd$ maps each reduce to the $\ppgd$ map. Furthermore, as seen in Section~\ref{subsec:StateOfTheArt}, the respective outputs of the $\pgd$ and $\ppgd$ maps are generated by a truncated SVD. Thus, the iterates of $\pgd$, $\ppgd$, $\ppgdr$, and $\ppgd$--$\pgd$ are naturally stored as compact SVDs; a \emph{compact} \cite{DemmelSVD} or \emph{reduced} \cite[Definition~2.27]{Hackbusch} SVD involves only the nonzero singular values and their associated singular vectors. Hence, the computational cost of the $\ppgdr$ map is merely the sum of those of the $\ppgd$ maps that it requires, and the computational cost of the $\ppgd$--$\pgd$ map is merely that of the $\ppgd$ or $\pgd$ map.

As explained next, $\rfdr$ and the proposed $\ppgdr$ and $\ppgd$--$\pgd$ compare favorably with the four other methods reviewed in Section~\ref{subsec:StateOfTheArt}, regarding the guarantee of accumulation at B-stationary points of~\eqref{eq:OptiDeterminantalVariety} (see Table~\ref{tab:OptimizationMethodsDeterminantalVariety}), the computational cost per iteration (see Section~\ref{subsec:ComparisonComputationalCostPerIteration}), and empirically observed numerical performance on two weighted low-rank approximation problems (see Sections~\ref{subsec:ComparisonWLRAapo} and~\ref{subsec:ComparisonMatrixCompletion}). However, in contrast with $\rfdr$, $\ppgdr$ and $\ppgd$--$\pgd$ can be defined on feasible sets for which no restricted tangent cone is known, such as the subset of $\R_{\le r}^{n \times n}$ containing the symmetric positive-semidefinite matrices, which appears notably in relaxations of combinatorial optimization problems; see, e.g., \cite{BurerMonteiro2003,BurerMonteiro2005,JourneeBachAbsilSepulchre2010,JiaEtAl} and references therein.

As seen in Section~\ref{subsec:StateOfTheArt}, $\pgd$ and $\hrtr$ each require a computationally prohibitive operation at every iteration, unless $f$ has some structure that can be exploited to reduce complexity.
In contrast, the other methods provably accumulating at B-stationary points of~\eqref{eq:OptiDeterminantalVariety}, namely $\rfdr$ and the proposed $\ppgdr$ and $\ppgd$--$\pgd$, require such an operation if and only if the $r$th singular value of their input is smaller than or equal to their parameter $\Delta \in (0, \infty)$, which should be rather infrequent in practice if $\Delta$ is chosen reasonably small (see Sections~\ref{subsec:ChoiceDelta}, \ref{subsec:ComparisonWLRAapo}, and~\ref{subsec:ComparisonMatrixCompletion}). Moreover, $\rfdr$ requires at most one such operation per iteration, in contrast with $\ppgdr$ and $\ppgd$--$\pgd$.

The six first-order methods listed in Table~\ref{tab:OptimizationMethodsDeterminantalVariety} are compared empirically on two weighted low-rank approximation problems in Sections~\ref{subsec:ComparisonWLRAapo} and~\ref{subsec:ComparisonMatrixCompletion}; $\hrtr$ was not included in the comparisons because it was at least $300$ times slower than the fastest method on one randomly generated instance of each problem. On the second problem, a matrix completion problem similar to that from \cite[\S 3.4]{SchneiderUschmajew2015}, $\ppgd$ and the proposed $\ppgdr$ and $\ppgd$--$\pgd$ outperform the other methods (see Section~\ref{subsec:ComparisonMatrixCompletion}). On the first problem, $\ppgd$ and $\rfd$ are outperformed by the other methods, and $\rfdr$ performs slightly better than $\ppgdr$ and $\ppgd$--$\pgd$, which themselves perform better than $\pgd$. Actually, on the one hundred randomly generated instances of the problem, $\ppgd$ and $\rfd$ follow an apocalypse on respectively twenty and all instances: they each generate a sequence in $\R_r^{m \times n}$ along which the B-stationarity measure~\eqref{eq:NormProjectionNegativeGradientOntoTangentCone} converges to zero and whose limit is M-stationary but not B-stationary. In contrast, the other methods each generate a sequence that converges to the global minimizer on every instance.

This paper improves and extends the early version of this work available in the unpublished e-print \cite{OlikierGallivanAbsil2022}.  Propositions~\ref{prop:MstationarityRestriction}--\ref{prop:LowerSemicontinuityMeasureMstationarity} in Section~\ref{sec:ElementsVariationalAnalysis} give important properties of B- and M-stationarity. A theoretical framework is developed in Section~\ref{sec:SufficientDescentMaps} that facilitates the analysis of the proposed $\ppgdr$ and $\ppgd$--$\pgd$ in Sections~\ref{sec:P2GDRdeterminantalVariety} and~\ref{sec:P2GD--PGD}. Finally, Section~\ref{sec:ComparisonWithStateOfTheArt} contains a systematic comparison of $\ppgdr$ and $\ppgd$--$\pgd$ with the state-of-the-art methods listed in Table~\ref{tab:OptimizationMethodsDeterminantalVariety} based on the computational cost per iteration and empirically observed numerical performance on two weighted low-rank approximation problems.

When it became public in the unpublished e-print \cite{OlikierGallivanAbsil2022}, $\ppgdr$ was the first method to answer positively the following question raised in the preliminary version \href{https://arxiv.org/abs/2107.03877v1}{arXiv:2107.03877v1} of \cite{LevinKileelBoumal2023}: ``Is there an algorithm running directly on $\R_{\le r}^{m \times n}$ that only uses first-order information about the cost function and which is guaranteed to converge to a B-stationary point?'' Since then, as mentioned above, $\rfdr$ was proposed and $\pgd$ was proven to also answer positively the question. However, $\ppgdr$ still compares favorably with those methods (and thus the second-order $\hrtr$), and its development and analysis serve as a basis for those of $\ppgd$--$\pgd$.

\subsection{Organization of the paper}
\label{subsec:OrganizationPaper}
This paper is organized as follows. Section~\ref{sec:ElementsVariationalAnalysis} reviews elements of variational analysis that are used throughout the paper. Section~\ref{sec:SufficientDescentMaps} presents the theoretical framework used to analyze $\ppgdr$ in Section~\ref{sec:P2GDRdeterminantalVariety} and $\ppgd$--$\pgd$ in Section~\ref{sec:P2GD--PGD}. In Section~\ref{sec:UpperBoundDistanceToDeterminantalVarietyFromTangent}, an upper bound on the distance to $\R_{\le r}^{m \times n}$ from each of its tangent lines is derived. This upper bound is used in Section~\ref{sec:P2GDmapDeterminantalVariety} to obtain a lower bound on the step size of the $\ppgd$ map. This lower bound plays an instrumental role in the analysis of $\ppgdr$ in Section~\ref{sec:P2GDRdeterminantalVariety}. $\ppgdr$, $\ppgd$--$\pgd$, and the state-of-the-art methods listed in Table~\ref{tab:OptimizationMethodsDeterminantalVariety} are compared in Section~\ref{sec:ComparisonWithStateOfTheArt}. Section~\ref{sec:Conclusion} gathers concluding remarks. Appendix~\ref{sec:PracticalImplementationHRTR} provides implementation details for $\hrtr$.

\section{Elements of variational analysis}
\label{sec:ElementsVariationalAnalysis}
For the convenience of the reader, this section reviews background material from variational analysis that is used throughout the paper. The only new results are Proposition~\ref{prop:DistanceToSetLowerSemicontinuous}, Lemma~\ref{lemma:ZeroOnDeterminantalVarietyImpliesZeroGradientOnSingularPart}, Propositions~\ref{prop:MstationarityRestriction} and~\ref{prop:NotClarkeRegularImpliesApocalyptic}, Corollary~\ref{coro:ApocalypticNotClarkeRegular}, and Propositions~\ref{prop:LowerSemicontinuityMeasureMstationarity} and~\ref{prop:DeterminantalVarietyNoSerendipitousPoint}. Propositions~\ref{prop:TriangleInequalityProjection} and \ref{prop:LocalDeltaRank} state basic facts and are proven for completeness.
Section~\ref{subsec:ConePolarityProjection} is about cones, polarity, and projection. Section~\ref{subsec:OuterLimitOuterContinuityCorrespondence} concerns the concepts of outer limit and outer semicontinuity of a correspondence. Section~\ref{subsec:TangentNormalCones} reviews tangent and normal cones. Section~\ref{subsec:StationarityMeasuresOfStationarity} focuses on the two notions of stationarity given in Definition~\ref{def:M/B-Stationarity} and the stationarity measures defined in \eqref{eq:NormProjectionNegativeGradientOntoTangentCone}--\eqref{eq:DistanceNegativeGradientToNormalCone}.

The material is presented for a general Euclidean vector space $\mathcal{E}$ with inner product $\ip{\cdot}{\cdot\cdot}$ and induced norm $\norm{\cdot}$. For every $x \in \mathcal{E}$ and $\rho \in (0, \infty)$, $\ball(x, \rho) \coloneq \{y \in \mathcal{E} \mid \norm{x-y} < \rho\}$ and $\ball[x, \rho] \coloneq \{y \in \mathcal{E} \mid \norm{x-y} \le \rho\}$ are the open and closed balls of center $x$ and radius $\rho$ in $\mathcal{E}$. Given $x \in \mathcal{E}$, the notation $\ball[x, 0] \coloneq \{x\}$ is also used. A subset $S$ of $\mathcal{E}$ is said to be \emph{locally closed} at $x \in \mathcal{E}$ if there exists $\rho \in (0, \infty)$ such that $S \cap \ball[x, \rho]$ is closed \cite[\S 1I]{RockafellarWets}.

In Sections~\ref{subsec:ConePolarityProjection}, \ref{subsec:TangentNormalCones}, and \ref{subsec:StationarityMeasuresOfStationarity}, specific facts are also provided about the case of interest in this work, namely $\mathcal{E} = \R^{m \times n}$ with the Frobenius inner product $\ip{X}{Y} \coloneq \tr Y^\tp X$.
For every $p, q \in \N$, $0_{p \times q}$ is the zero matrix in $\R^{p \times q}$, $I_p$ is the identity matrix in $\R^{p \times p}$, and, if $q \ge p$, then
\begin{equation*}
\st(p, q) \coloneq \left\{U \in \R^{q \times p} \mid U^\tp U = I_p\right\}
\end{equation*}
is a Stiefel manifold \cite[\S 3.3.2]{AbsilMahonySepulchre}.

\subsection{Cone, polarity, and projection}
\label{subsec:ConePolarityProjection}
A nonempty subset $K$ of $\mathcal{E}$ is called a \emph{cone} if $x \in K$ implies $\alpha x \in K$ for all $\alpha \in [0, \infty)$ \cite[\S 3B]{RockafellarWets}.
For every cone $K \subseteq \mathcal{E}$, the set
\begin{equation*}
K^* \coloneq \{w \in \mathcal{E} \mid \ip{v}{w} \le 0 \; \forall v \in K\}
\end{equation*}
is a closed convex cone called the \emph{polar} of $K$ \cite[6(14)]{RockafellarWets}.

For every nonempty subset $S$ of $\mathcal{E}$ and every $x \in \mathcal{E}$, $\dist(x, S) \coloneq \inf_{y \in S} \norm{x-y}$ is the distance from $x$ to $S$, and $\proj{S}{x} \coloneq \argmin_{y \in S} \norm{x-y}$ is the projection of $x$ onto $S$; the set $\proj{S}{x}$ is nonempty and compact if $S$ is closed \cite[Example~1.20]{RockafellarWets}. In what follows, a singleton is sometimes identified with its unique element.

Proposition~\ref{prop:TriangleInequalityProjection} is invoked in the proof of Lemma~\ref{lemma:P2GDRmapPolak}.

\begin{proposition}
\label{prop:TriangleInequalityProjection}
Let $S$ be a nonempty subset of $\mathcal{E}$. For all $x \in S$ and $y \in \mathcal{E}$, $\proj{S}{y} \subseteq \ball[x, 2\norm{x-y}]$.
\end{proposition}

\begin{proof}
Let $x \in S$ and $y \in \mathcal{E}$. The inclusion is trivially true if $\proj{S}{y}$ is empty. If $\proj{S}{y}$ is nonempty, then, for all $z \in \proj{S}{y}$, $\norm{x-z} \le \norm{x-y} + \norm{y-z} = \norm{x-y} + \dist(y, S) \le 2 \norm{x-y}$.
\end{proof}

Proposition~\ref{prop:ProjectionOntoClosedCone} shows that the projection onto a closed cone enjoys an orthogonality property.

\begin{proposition}[{\cite[Proposition~A.6]{LevinKileelBoumal2023}}]
\label{prop:ProjectionOntoClosedCone}
Let $K \subseteq \mathcal{E}$ be a closed cone. For all $x \in \mathcal{E}$ and $y \in \proj{K}{x}$,
\begin{equation*}
\ip{x}{y} = \norm{y}^2
\end{equation*}
and thus
\begin{equation*}
\norm{y}^2 = \norm{x}^2 - \dist(x, K)^2.
\end{equation*}
\end{proposition}

The proposed $\ppgdr$ exploits that, according to the Eckart--Young theorem \cite{EckartYoung1936}, projecting onto the closed cone $\R_{\le r}^{m \times n}$ amounts to computing a truncated SVD. The rest of this section reviews related results that are used in Sections~\ref{sec:UpperBoundDistanceToDeterminantalVarietyFromTangent} and \ref{sec:P2GDRdeterminantalVariety}.

The singular values of $X \in \R^{m \times n}$ are denoted by $\sigma_1(X) \ge \dots \ge \sigma_{\min\{m, n\}}(X) \ge 0$. Moreover, if $X \ne 0_{m \times n}$, then $\sigma_{\rank X}(X)$ is denoted by $\sigma_{\min}(X)$. The spectral norm of $X \in \R^{m \times n}$, denoted by $\norm{X}_2$, equals $\sigma_1(X)$ \cite[Example~5.6.6]{HornJohnson}.
Proposition~\ref{prop:SingularValuesLipschitz} shows that the singular values are Lipschitz continuous with Lipschitz constant $1$, a property that is used to prove Proposition~\ref{prop:LocalDeltaRank} and Lemma~\ref{lemma:P2GDRmapPolak}.

\begin{proposition}[{\cite[Corollary~8.6.2]{GolubVanLoan}}]
\label{prop:SingularValuesLipschitz}
For all $X, Y \in \R^{m \times n}$ and $i \in \{1, \dots, \min\{m, n\}\}$,
\begin{equation*}
|\sigma_i(X)-\sigma_i(Y)| \le \norm{X-Y}_2 \le \norm{X-Y}.
\end{equation*}
\end{proposition}

Proposition~\ref{prop:UpperBoundFrobeniusNormProduct} is invoked in the proof of Proposition~\ref{prop:GlobalSecondOrderUpperBoundDistanceToRealDeterminantalVarietyFromTangentLine}.

\begin{proposition}[{\cite[Lemma~2.11]{Hackbusch}}]
\label{prop:UpperBoundFrobeniusNormProduct}
For all real matrices $X$ and $Y$ such that the number of columns of $X$ equals the number of rows of $Y$,
\begin{equation*}
\norm{XY} \le \norm{X}_2 \norm{Y}.
\end{equation*}
\end{proposition}

By reducing the rank of $X \in \R^{m \times n} \setminus \{0_{m \times n}\}$, we mean computing an element of $\proj{\R_{\le \ushort{r}}^{m \times n}}{X}$ for some nonnegative integer $\ushort{r} < \rank X$, with the convention $\R_{\le 0}^{m \times n} \coloneq \R_0^{m \times n} \coloneq \{0_{m \times n}\}$. The Eckart--Young theorem implies that for every nonnegative integer $\ushort{r} < \min\{m, n\}$ and $X \in \R^{m \times n}$:
\begin{enumerate}
\item if $\rank X \le \ushort{r}$, then $\proj{\R_{\le \ushort{r}}^{m \times n}}{X} = \{X\}$;
\item if $\rank X > \ushort{r}$, then $\dist(X, \R_{\le \ushort{r}}^{m \times n}) = \dist(X, \R_{\ushort{r}}^{m \times n}) = \sqrt{\sum_{i=\ushort{r}+1}^{\rank X} \sigma_i^2(X)}$ and $\proj{\R_{\le \ushort{r}}^{m \times n}}{X} = \proj{\R_{\ushort{r}}^{m \times n}}{X}$.
\end{enumerate}

Proposition~\ref{prop:LocalDeltaRank} is invoked in the proof of Lemma~\ref{lemma:P2GDRmapPolak}.

\begin{proposition}
\label{prop:LocalDeltaRank}
Let $\ushort{X} \in \R^{m \times n}$ and $\ushort{r} \coloneq \rank \ushort{X}$. For all $\Delta \in (0, \infty)$ and $X \in \ball[\ushort{X}, \Delta]$, it holds that $\rank_\Delta X \le \ushort{r}$. If $\ushort{X} \ne 0_{m \times n}$, then, for all $X \in \ball(\ushort{X}, \sigma_{\ushort{r}}(\ushort{X}))$, it holds that $\ushort{r} \le \rank X$.
\end{proposition}

\begin{proof}
The first inequality follows from Proposition~\ref{prop:SingularValuesLipschitz}:
\begin{equation*}
\sigma_{\ushort{r}+1}(X)
= |\sigma_{\ushort{r}+1}(X)-\sigma_{\ushort{r}+1}(\ushort{X})|
\le \norm{X-\ushort{X}}
\le \Delta.
\end{equation*}
The second inequality holds because $\ball(\ushort{X}, \sigma_{\ushort{r}}(\ushort{X})) \subseteq \R^{m \times n} \setminus \R_{< \ushort{r}}^{m \times n}$ since $\sigma_{\ushort{r}}(\ushort{X}) = \dist(\ushort{X}, \R_{< \ushort{r}}^{m \times n})$.
\end{proof}

\subsection{Outer limit and outer semicontinuity of a correspondence}
\label{subsec:OuterLimitOuterContinuityCorrespondence}
The material reviewed in this section is used in Sections~\ref{subsec:TangentNormalCones} and \ref{subsec:StationarityMeasuresOfStationarity}.

Let $(S_i)_{i \in \N}$ be a sequence of nonempty subsets of $\mathcal{E}$. The \emph{outer limit} of $(S_i)_{i \in \N}$, denoted by $\outlim_{i \to \infty} S_i$, is the set of all possible accumulation points of sequences $(x_i)_{i \in \N}$ such that $x_i \in S_i$ for all $i \in \N$ \cite[\S 4A]{RockafellarWets}.

A \emph{correspondence}, or a \emph{set-valued map}, is a triplet $F \coloneq (A_1, A_2, G)$ where $A_1$ and $A_2$ are sets respectively called the \emph{set of departure} and the \emph{set of destination} of $F$, and $G$ is a subset of $A_1 \times A_2$ called the \emph{graph} of $F$. If $F \coloneq (A_1, A_2, G)$ is a correspondence, written $F : A_1 \setmapsto A_2$, then the \emph{image} of $x \in A_1$ under $F$ is $F(x) \coloneq \{y \in A_2 \mid (x, y) \in G\}$ and the \emph{domain} of $F$ is $\dom F \coloneq \{x \in A_1 \mid F(x) \ne \emptyset\}$.

Let $\mathcal{E}_1$ and $\mathcal{E}_2$ be two Euclidean vector spaces, $F : \mathcal{E}_1 \setmapsto \mathcal{E}_2$ be a correspondence, $S$ be a nonempty subset of $\dom F$, and $\ushort{x} \in S$.
The \emph{outer limit} of $F$ at $\ushort{x}$ relative to $S$, denoted by $\outlim_{S \ni x \to \ushort{x}} F(x)$, is the set of all $y \in \mathcal{E}_2$ such that there exist sequences $(x_i)_{i \in \N}$ in $S$ converging to $\ushort{x}$ and $(y_i)_{i \in \N}$ in $\mathcal{E}_2$ converging to $y$ such that $y_i \in F(x_i)$ for all $i \in \N$ \cite[5(1)]{RockafellarWets}.
The correspondence $F$ is said to be \emph{outer semicontinuous} at $\ushort{x}$ relative to $S$ if $\outlim_{S \ni x \to \ushort{x}} F(x) \subseteq F(\ushort{x})$ or, equivalently, $\outlim_{S \ni x \to \ushort{x}} F(x) = F(\ushort{x})$ \cite[Definition~5.4]{RockafellarWets}.

Proposition~\ref{prop:DistanceToSetLowerSemicontinuous} enables to prove Proposition~\ref{prop:LowerSemicontinuityMeasureMstationarity}.

\begin{proposition}
\label{prop:DistanceToSetLowerSemicontinuous}
Let $g : \mathcal{E} \to \mathcal{E}$ be continuous, $F : \mathcal{E} \setmapsto \mathcal{E}$ be closed-valued, and $S$ be a nonempty subset of $\dom F$. If $F$ is outer semicontinuous at $\ushort{x} \in S$ relative to $S$, then the function
\begin{equation*}
\dom F \to \R : x \mapsto \dist(g(x), F(x))
\end{equation*}
is lower semicontinuous at $\ushort{x}$ relative to $S$.
\end{proposition}

\begin{proof}
For all $x \in \dom F$, $\dist(g(x), F(x)) = \dist(0, F(x)-g(x))$. Let $F$ be outer semicontinuous at $\ushort{x} \in S$ relative to $S$. By \cite[Exercise~5.24]{RockafellarWets}, since $-g$ is continuous, and thus locally bounded, $F-g$ is outer semicontinuous at $\ushort{x}$ relative to $S$. Thus, by \cite[Proposition~5.11(a)]{RockafellarWets}, the function
\begin{equation*}
\dom F \to \R : x \mapsto \dist(0, F(x)-g(x))
\end{equation*}
is lower semicontinuous at $\ushort{x}$ relative to $S$.
\end{proof}

\subsection{Tangent and normal cones}
\label{subsec:TangentNormalCones}
Based on \cite[Chapter~6]{RockafellarWets}, this section reviews the concept of tangent cone and the two notions of normality on which the two notions of stationarity given in Definition~\ref{def:M/B-Stationarity} rely.
Let $S$ be a nonempty subset of $\mathcal{E}$.

A vector $v \in \mathcal{E}$ is said to be \emph{tangent} to $S$ at $x \in S$ if there exist sequences $(x_i)_{i \in \N}$ in $S$ converging to $x$ and $(t_i)_{i \in \N}$ in $(0, \infty)$ such that the sequence $(\frac{x_i-x}{t_i})_{i \in \N}$ converges to $v$ \cite[Definition~6.1]{RockafellarWets}. The set of all tangent vectors to $S$ at $x \in S$ is a closed cone \cite[Proposition~6.2]{RockafellarWets} called the \emph{(Bouligand) tangent cone} to $S$ at $x$ and denoted by $\tancone{S}{x}$.
The \emph{regular normal cone} to $S$ at $x \in S$ is the polar
\begin{equation*}
\regnorcone{S}{x} \coloneq \tancone{S}{x}^*,
\end{equation*}
and the \emph{normal cone} to $S$ at $x$ is the outer limit
\begin{equation*}
\norcone{S}{x} \coloneq \outlim_{S \ni z \to x} \regnorcone{S}{z},
\end{equation*}
which is a closed cone \cite[Definition~6.3 and Proposition~6.5]{RockafellarWets}.
For every $x \in S$,
\begin{equation}
\label{eq:NestedNormalCones}
\regnorcone{S}{x} \subseteq \norcone{S}{x},
\end{equation}
and $S$ is said to be \emph{Clarke regular} at $x$ if this inclusion is an equality and $S$ is locally closed at $x$ \cite[Definition~6.4]{RockafellarWets}.

If $S$ is an embedded submanifold of $\mathcal{E}$ around $x \in S$ \cite[\S 3.3.1]{AbsilMahonySepulchre}, then $S$ is Clarke regular at $x$, $\tancone{S}{x}$ and $\norcone{S}{x}$ coincide respectively with the tangent and normal spaces to $S$ at $x$ \cite[Example~6.8]{RockafellarWets}, which are the orthogonal complements of each other, and, for every differentiable function $\varphi : \mathcal{E} \to \R$, the Riemannian gradient of $\varphi|_S$ at $x$ equals $\proj{\tancone{S}{x}}{\nabla\varphi(x)}$ \cite[\S 3.6.1]{AbsilMahonySepulchre}. (The set $C$ in \cite[Example~6.8]{RockafellarWets} is an embedded submanifold of $\R^n$ locally around $\bar{x}$, as can be seen by applying the submersion theorem \cite[Proposition~3.3.3]{AbsilMahonySepulchre} with $\mathcal{M}_1 \coloneq O$, $d_1 \coloneq n$, $\mathcal{M}_2 \coloneq \R^m$, $d_2 \coloneq m$, and $y \coloneq 0$.)

For every $X \in \R_{\le r}^{m \times n}$, the closed cones $\tancone{\R_{\le r}^{m \times n}}{X}$, $\restancone{\R_{\le r}^{m \times n}}{X}$, $\regnorcone{\R_{\le r}^{m \times n}}{X}$, and $\norcone{\R_{\le r}^{m \times n}}{X}$ can be described explicitly based on orthonormal bases of $\im X$, $\im X^\tp$, and their orthogonal complements. Let $\ushort{r} \coloneq \rank X$, $U \in \st(\ushort{r}, m)$, $U_\perp \in \st(m-\ushort{r}, m)$, $V \in \st(\ushort{r}, n)$, $V_\perp \in \st(n-\ushort{r}, n)$, $\im U = \im X$, $\im U_\perp = (\im X)^\perp$, $\im V = \im X^\tp$, and $\im V_\perp = (\im X^\tp)^\perp$.
Then:
\begin{itemize}
\item by \cite[Theorem~3.2]{SchneiderUschmajew2015} (see also \cite[Theorem~4.2]{OlikierMlinaricAbsilUschmajew}) and \cite[Definition~3.1]{OlikierAbsil2023},
\begin{align}
\label{eq:TanConeDeterminantalVariety}
\tancone{\R_{\le r}^{m \times n}}{X}
&= [U \; U_\perp] \begin{bmatrix} \R^{\ushort{r} \times \ushort{r}} & \R^{\ushort{r} \times n-\ushort{r}} \\ \R^{m-\ushort{r} \times \ushort{r}} & \R_{\le r-\ushort{r}}^{m-\ushort{r} \times n-\ushort{r}} \end{bmatrix} [V \; V_\perp]^\tp,\\
\nonumber
\restancone{\R_{\le r}^{m \times n}}{X}
&= \left\{[U \; U_\perp] \begin{bmatrix} A_{1, 1} & A_{1, 2} \\ A_{2, 1} & A_{2, 2} \end{bmatrix} [V \; V_\perp]^\tp ~\left|~\begin{array}{l} A_{1, 1} \in \R^{\ushort{r} \times \ushort{r}},\, A_{1, 2} \in \R^{\ushort{r} \times n-\ushort{r}},\\ A_{2, 1} \in \R^{m-\ushort{r} \times \ushort{r}},\, A_{2, 2} \in \R_{\le r-\ushort{r}}^{m-\ushort{r} \times n-\ushort{r}}, \\ A_{1, 2} = 0_{\ushort{r} \times n-\ushort{r}} \text{ or } A_{2, 1} = 0_{m-\ushort{r} \times \ushort{r}} \end{array}\right.\right\}
\end{align}
and, by \cite[Corollary~3.3]{SchneiderUschmajew2015} and \cite[Proposition~3.2]{OlikierAbsil2023}, for every $Z \in \R^{m \times n}$ written as
\begin{equation*}
Z = [U \; U_\perp] \begin{bmatrix} A_{1, 1} & A_{1, 2} \\ A_{2, 1} & A_{2, 2} \end{bmatrix} [V \; V_\perp]^\tp
\end{equation*}
with $A_{1, 1} = U^\tp Z V$, $A_{1, 2} = U^\tp Z V_\perp$, $A_{2, 1} = U_\perp^\tp Z V$, and $A_{2, 2} = U_\perp^\tp Z V_\perp$,
\begin{align}
\label{eq:ProjTanConeDeterminantalVariety}
\proj{\tancone{\R_{\le r}^{m \times n}}{X}}{Z}
&= [U \; U_\perp] \begin{bmatrix} A_{1, 1} & A_{1, 2} \\ A_{2, 1} & \proj{\R_{\le r-\ushort{r}}^{m-\ushort{r} \times n-\ushort{r}}}{A_{2, 2}} \end{bmatrix} [V \; V_\perp]^\tp,\\
\label{eq:NormProjTanConeDeterminantalVariety}
\norm{\proj{\tancone{\R_{\le r}^{m \times n}}{X}}{Z}}
&\ge \sqrt{\frac{r-\ushort{r}}{\min\{m,n\}-\ushort{r}}} \norm{Z},\\
\label{eq:ProjResTanConeDeterminantalVariety}
\proj{\restancone{\R_{\le r}^{m \times n}}{X}}{Z}
&= \left\{\begin{array}{ll}
\begin{array}{l}
[U \; U_\perp] \begin{bmatrix} A_{1, 1} & A_{1, 2} \\ 0_{m-\ushort{r} \times \ushort{r}} & \proj{\R_{\le r-\ushort{r}}^{m-\ushort{r} \times n-\ushort{r}}}{A_{2, 2}} \end{bmatrix} [V \; V_\perp]^\tp
\end{array} & \text{if } \norm{A_{1, 2}} > \norm{A_{2, 1}},\\[5mm]
\begin{array}{l}
[U \; U_\perp] \begin{bmatrix} A_{1, 1} & A_{1, 2} \\ 0_{m-\ushort{r} \times \ushort{r}} & \proj{\R_{\le r-\ushort{r}}^{m-\ushort{r} \times n-\ushort{r}}}{A_{2, 2}} \end{bmatrix} [V \; V_\perp]^\tp\\[5mm]
\cup~[U \; U_\perp] \begin{bmatrix} A_{1, 1} & 0_{\ushort{r} \times n-\ushort{r}} \\ A_{2, 1} & \proj{\R_{\le r-\ushort{r}}^{m-\ushort{r} \times n-\ushort{r}}}{A_{2, 2}} \end{bmatrix} [V \; V_\perp]^\tp
\end{array} & \text{if } \norm{A_{1, 2}} = \norm{A_{2, 1}},\\[13mm]
\begin{array}{l}
[U \; U_\perp] \begin{bmatrix} A_{1, 1} & 0_{\ushort{r} \times n-\ushort{r}} \\ A_{2, 1} & \proj{\R_{\le r-\ushort{r}}^{m-\ushort{r} \times n-\ushort{r}}}{A_{2, 2}} \end{bmatrix} [V \; V_\perp]^\tp
\end{array} & \text{if } \norm{A_{1, 2}} < \norm{A_{2, 1}};
\end{array}\right.
\end{align}

\item by \cite[Theorem~3.1]{HosseiniLukeUschmajew2019},
\begin{equation}
\label{eq:NorConeDeterminantalVariety}
\norcone{\R_{\le r}^{m \times n}}{X} = U_\perp \R_{\le \min\{m, n\}-r}^{m-\ushort{r} \times n-\ushort{r}} V_\perp^\tp;
\end{equation}

\item by \cite[Corollary~2.3]{HosseiniLukeUschmajew2019}, if $\ushort{r} < r$, then
\begin{equation}
\label{eq:RegNorConeDeterminantalVariety}
\regnorcone{\R_{\le r}^{m \times n}}{X} = \{0_{m \times n}\}.
\end{equation}
\end{itemize}
As pointed out in Section~\ref{sec:Introduction}, some of these equations are known in algebraic geometry since the last century. To exploit \eqref{eq:TanConeDeterminantalVariety} and \eqref{eq:ProjTanConeDeterminantalVariety} efficiently, note the following.
\begin{enumerate}
\item In practice, the projection onto $\tancone{\R_{\le r}^{m \times n}}{X}$ can be computed by \cite[Algorithm~2]{SchneiderUschmajew2015}. This does not rely on $U_\perp$ and $V_\perp$, which are huge in the frequently encountered case where $r \ll \min\{m,n\}$.

\item As pointed out in Section~\ref{subsec:StateOfTheArt}, for all $X \in \R_{\le r}^{m \times n}$ and $Z \in \tancone{\R_{\le r}^{m \times n}}{X}$, $X+Z \in \R_{\le 2r}^{m \times n}$.
(In what follows, the $\diag$ operator returns a block diagonal matrix \cite[\S 1.3.1]{GolubVanLoan} whose diagonal blocks are the arguments given to the operator.)
Indeed, if
\begin{equation*}
X = [U \; U_\perp] \diag(\Sigma, 0_{m-\ushort{r} \times n-\ushort{r}}) [V \; V_\perp]^\tp
\end{equation*}
is an SVD and $Z \in \tancone{\R_{\le r}^{m \times n}}{X}$ is written as
\begin{equation*}
Z = [U \; U_\perp] \begin{bmatrix} A_{1, 1} & A_{1, 2} \\ A_{2, 1} & A_{2, 2} \end{bmatrix} [V \; V_\perp]^\tp
\end{equation*}
with $A_{1, 1} \in \R^{\ushort{r} \times \ushort{r}}$, $A_{1, 2} \in \R^{\ushort{r} \times n-\ushort{r}}$, $A_{2, 1} \in \R^{m-\ushort{r} \times \ushort{r}}$, and $A_{2, 2} \in \R_{\le r-\ushort{r}}^{m-\ushort{r} \times n-\ushort{r}}$, then, by \cite[Proposition~3.1]{OlikierAbsil2022},
\begin{equation*}
X+Z = [U \; U_\perp] \begin{bmatrix} \Sigma+A_{1, 1} & A_{1, 2} \\ A_{2, 1} & A_{2, 2} \end{bmatrix} [V \; V_\perp]^\tp \in \R_{\le r+\ushort{r}}^{m \times n} \subseteq \R_{\le 2r}^{m \times n}.
\end{equation*}
This is exploited in the analysis of the computational cost per iteration of $\ppgd$ and $\ppgdr$ conducted in \cite[\S 7]{OlikierAbsil2023}.
\end{enumerate}

\subsection{Stationarity and measures of stationarity}
\label{subsec:StationarityMeasuresOfStationarity}
This section completes Section~\ref{subsec:StationarityNotionsDeterminantalVariety} by establishing several results mentioned therein.

For a set that is not Clarke regular, M-stationarity is weaker than B-stationarity. For problem~\eqref{eq:OptiDeterminantalVariety}, this is quantified in Remark~\ref{rem:MstationarityDeterminantalVariety}.

\begin{remark}
\label{rem:MstationarityDeterminantalVariety}
The formulas \eqref{eq:NorConeDeterminantalVariety} and \eqref{eq:RegNorConeDeterminantalVariety} show that, for problem~\eqref{eq:OptiDeterminantalVariety}, M-stationarity is weaker than B-stationarity at every $X \in \R_{\ushort{r}}^{m \times n}$ with $\ushort{r} \in \{0, \dots, r-1\}$. Indeed, while the latter amounts to $\nabla f(X) = 0_{m \times n}$, the former only amounts to the B-stationarity of $X$ for the problem of minimizing $f$ on $\R_{\le \ushort{r}}^{m \times n}$ together with the inequality $\rank \nabla f(X) \le \min\{m,n\}-r$.
\end{remark}

Proposition~\ref{prop:MstationarityRestriction} states that, for problem~\eqref{eq:OptiDeterminantalVariety}, the M-stationarity of a point depends on $f$ only through its restriction $f|_{\R_{\le r}^{m \times n}}$. It is based on Lemma~\ref{lemma:ZeroOnDeterminantalVarietyImpliesZeroGradientOnSingularPart}.

\begin{lemma}
\label{lemma:ZeroOnDeterminantalVarietyImpliesZeroGradientOnSingularPart}
For every differentiable function $g : \R^{m \times n} \to \R$ such that $g(\R_{\le r}^{m \times n}) = \{0\}$, it holds that $\nabla g(\R_{< r}^{m \times n}) = \{0_{m \times n}\}$.
\end{lemma}

\begin{proof}
Let $g : \R^{m \times n} \to \R$ be differentiable and such that $g(\R_{\le r}^{m \times n}) = \{0\}$. Let $X \in \R_{< r}^{m \times n}$.
For every $(i, j) \in \{1, \dots, m\} \times \{1, \dots, n\}$, let $E_{i, j} \coloneq [\delta_{i, k} \delta_{j, l}]_{k, l = 1}^{m, n}$ and $\partial_{i, j} g(X) \coloneq \lim_{t \to 0} \frac{g(X+tE_{i, j})-g(X)}{t}$. The numerator in the limit is zero since, for every $t \in \R$, $\rank(X+tE_{i, j}) \le \rank X + \rank tE_{i, j} \le r$. The result follows since $\nabla g(X) = [\partial_{i, j} g(X)]_{i, j = 1}^{m, n}$.
\end{proof}

\begin{proposition}
\label{prop:MstationarityRestriction}
Let $h : \R^{m \times n} \to \R$ be differentiable and such that $h|_{\R_{\le r}^{m \times n}} = f|_{\R_{\le r}^{m \times n}}$. Then, $X \in \R_{\le r}^{m \times n}$ is M-stationary for the minimization of $h|_{\R_{\le r}^{m \times n}}$ if and only if $X$ is M-stationary for~\eqref{eq:OptiDeterminantalVariety}.
\end{proposition}

\begin{proof}
Let $X \in \R_{\le r}^{m \times n}$.
If $X \in \R_r^{m \times n}$, the result follows from the equality $\norcone{\R_{\le r}^{m \times n}}{X} = \regnorcone{\R_{\le r}^{m \times n}}{X}$ and the fact that the B-stationarity of $X$ depends on $f$ only through $f|_{\R_{\le r}^{m \times n}}$ \cite[Appendix~A]{LevinKileelBoumal2023}.
If $X \in \R_{< r}^{m \times n}$, Lemma~\ref{lemma:ZeroOnDeterminantalVarietyImpliesZeroGradientOnSingularPart} implies that $\nabla h(X) = \nabla f(X)$, and the result follows.
\end{proof}

The rest of this section concerns the measures of stationarity in~\eqref{eq:NormProjectionNegativeGradientOntoTangentCone}--\eqref{eq:DistanceNegativeGradientToNormalCone}.
Let $C$ be a nonempty closed subset of $\mathcal{E}$.

By Proposition~\ref{prop:ProjectionOntoClosedCone}, for every differentiable function $\varphi : \mathcal{E} \to \R$ and every $x \in C$,
\begin{equation}
\label{eq:StationarityMeasureExplicitFormula}
\norm{\proj{\tancone{C}{x}}{-\nabla \varphi(x)}}
= \sqrt{\norm{\nabla \varphi(x)}^2 - \dist(-\nabla \varphi(x), \tancone{C}{x})^2}
\eqcolon \s(x; \varphi, C).
\end{equation}
The stationarity measure $\s(\cdot; \varphi, C)$ is used in \cite{SchneiderUschmajew2015} to define and analyze $\ppgd$ and $\rfd$, and in this paper to analyze $\ppgdr$. It possesses the following desirable properties:
\begin{enumerate}
\item the map that associates with every $x \in C$ the set $\proj{\tancone{C}{x}}{-\nabla \varphi(x)}$ depends on $\varphi$ only through its restriction $\varphi|_C$ \cite[Lemmas~A.7 and A.8]{LevinKileelBoumal2023};
\item for every $x \in C$, $-\nabla \varphi(x) \in \regnorcone{C}{x}$ if and only if $\s(x; \varphi, C) = 0$ (see \cite[\S 2.1]{SchneiderUschmajew2015} or \cite[Proposition~2.5]{LevinKileelBoumal2023});
\item for every $x \in C$, $\s(x; \varphi, C)$ is the best linear rate of decrease that can achieved in $\varphi$ by moving away from $x$ along a direction in $\tancone{C}{x}$ \cite[Theorem~A.9(b) in \href{https://arxiv.org/abs/2107.03877v3}{arXiv:2107.03877v3}]{LevinKileelBoumal2023}.
\end{enumerate}

Proposition~\ref{prop:NotClarkeRegularImpliesApocalyptic} implies that the measures of B-stationarity in~\eqref{eq:NormProjectionNegativeGradientOntoTangentCone} and \eqref{eq:DistanceNegativeGradientToRegularNormalCone} can fail to be lower semicontinuous at every point where Clarke regularity is absent.

\begin{proposition}
\label{prop:NotClarkeRegularImpliesApocalyptic}
For every $x \in C$ where $C$ is not Clarke regular, there exist a sequence $(x_i)_{i \in \N}$ in $C$ converging to $x$ and a continuously differentiable function $\varphi : \mathcal{E} \to \R$ such that the measures of B-stationarity $\s(\cdot; \varphi, C)$ and
\begin{equation*}
C \to \R : x \mapsto \dist(-\nabla \varphi(x), \regnorcone{C}{x})
\end{equation*}
converge to zero along $(x_i)_{i \in \N}$, yet $-\nabla \varphi(x) \in \norcone{C}{x} \setminus \regnorcone{C}{x}$.
\footnote{We are grateful to an anonymous reviewer of \href{https://arxiv.org/abs/2303.16040v1}{arXiv:2303.16040v1} for pointing this out.}
\end{proposition}

\begin{proof}
Assume that $C$ is not Clarke regular at $x \in C$. Let $v \in \norcone{C}{x} \setminus \regnorcone{C}{x}$ and $\varphi : \mathcal{E} \to \R : z \mapsto -\ip{z}{v}$. Then, $\varphi$ is continuously differentiable and $-\nabla \varphi(z) = v$ for all $z \in \mathcal{E}$. In particular, $\s(x; \varphi, C) \ne 0$. Furthermore, by definition of $\norcone{C}{x}$, there exist sequences $(x_i)_{i \in \N}$ in $C$ converging to $x$ and $(v_i)_{i \in \N}$ converging to $v$ such that $v_i \in \regnorcone{C}{x_i}$ for all $i \in \N$. Therefore, for all $i \in \N$,
\begin{equation*}
\s(x_i; \varphi, C)
\le \dist(-\nabla \varphi(x_i), \regnorcone{C}{x_i})
\le \norm{v-v_i}
\to 0 \text{ when } i \to \infty,
\end{equation*}
where the first inequality follows from \cite[Theorem~A.9(c) in \href{https://arxiv.org/abs/2107.03877v3}{arXiv:2107.03877v3}]{LevinKileelBoumal2023}.
\end{proof}

The stationarity measure $\s(\cdot; \varphi, C)$ is used in \cite{LevinKileelBoumal2023} to define the concept of apocalypse. A point $x \in C$ is said to be \emph{apocalyptic} if there exist a sequence $(x_i)_{i \in \N}$ in $C$ converging to $x$ and a continuously differentiable function $\varphi : \mathcal{E} \to \R$ such that $\lim_{i \to \infty} \s(x_i; \varphi, C) = 0$ whereas $\s(x; \varphi, C) > 0$ \cite[Definition~2.7]{LevinKileelBoumal2023}. Such a triplet $(x, (x_i)_{i \in \N}, \varphi)$ is called an \emph{apocalypse}.

Corollary~\ref{coro:ApocalypticNotClarkeRegular} answers a question raised in \cite[\S 4]{LevinKileelBoumal2023}.

\begin{corollary}
\label{coro:ApocalypticNotClarkeRegular}
A point $x \in C$ is apocalyptic if and only if $C$ is not Clarke regular at $x$.
\end{corollary}

\begin{proof}
This follows from \cite[Corollary~2.15]{LevinKileelBoumal2023} and Proposition~\ref{prop:NotClarkeRegularImpliesApocalyptic}.
\end{proof}

On a set that is not Clarke regular, no measure of B-stationarity in the literature is known to be lower semicontinuous. In contrast, the measure of M-stationarity in~\eqref{eq:DistanceNegativeGradientToNormalCone} is lower semicontinuous.

\begin{proposition}
\label{prop:LowerSemicontinuityMeasureMstationarity}
For every continuously differentiable function $\varphi : \mathcal{E} \to \R$, the function
\begin{equation*}
C \to \R : x \mapsto \dist(-\nabla\varphi(x), \norcone{C}{x})
\end{equation*}
is lower semicontinuous.
\end{proposition}

\begin{proof}
This follows from Proposition~\ref{prop:DistanceToSetLowerSemicontinuous} since $N_C$ is outer semicontinuous relative to $C$ \cite[Proposition~6.6]{RockafellarWets}.
\end{proof}

The notion of serendipity is complementary to that of apocalypse. A point $x \in C$ is said to be \emph{serendipitous} if there exist a sequence $(x_i)_{i \in \N}$ in $C$ converging to $x$, a continuously differentiable function $\varphi : \mathcal{E} \to \R$, and $\varepsilon \in (0, \infty)$ such that $\s(x_i; \varphi, C) > \varepsilon$ for all $i \in \N$, yet $\s(x; \varphi, C) = 0$ \cite[Definition~2.8]{LevinKileelBoumal2023}.
Thus, if $x \in C$ is not serendipitous and $\s(x; \varphi, C) = 0$ for some continuously differentiable function $\varphi : \mathcal{E} \to \R$, then $\lim_{i \to \infty} \s(x_i; \varphi, C) = 0$ for every sequence $(x_i)_{i \in \N}$ in $C$ converging to $x$.
A point $x \in C$ is serendipitous if and only if there exists a sequence $(x_i)_{i \in \N}$ in $C$ converging to $x$ such that $\regnorcone{C}{x} \not\subseteq \left(\outlim_{i \to \infty} \tancone{C}{x_i}\right)^*$ \cite[Theorem~2.17]{LevinKileelBoumal2023}.
Proposition~\ref{prop:DeterminantalVarietyNoSerendipitousPoint} enables to deduce Corollary~\ref{coro:P2GDRPolak} from Theorem~\ref{thm:P2GDRPolak}.

\begin{proposition}
\label{prop:DeterminantalVarietyNoSerendipitousPoint}
No point of $\R_{\le r}^{m \times n}$ is serendipitous.
\end{proposition}

\begin{proof}
We use \cite[Theorem~2.17]{LevinKileelBoumal2023}. Let $X \in \R_{\le r}^{m \times n}$. Let us prove that $X$ is not serendipitous. Let $(X_i)_{i \in \N}$ be a sequence in $\R_{\le r}^{m \times n}$ converging to $X$.
If $\rank X = r$, then $\outlim_{i \to \infty} \tancone{\R_{\le r}^{m \times n}}{X_i} = \tancone{\R_{\le r}^{m \times n}}{X}$ \cite[Proposition~4.3]{OlikierAbsil2022}, and thus $\big(\outlim_{i \to \infty} \tancone{\R_{\le r}^{m \times n}}{X_i}\big)^* = \regnorcone{\R_{\le r}^{m \times n}}{X}$.
If $\rank X < r$, then
\begin{equation*}
\regnorcone{\R_{\le r}^{m \times n}}{X}
= \{0_{m \times n}\}
\subseteq \Big(\outlim_{i \to \infty} \tancone{\R_{\le r}^{m \times n}}{X_i}\Big)^*,
\end{equation*}
where the equality follows from \eqref{eq:RegNorConeDeterminantalVariety}.
\end{proof}

Proposition~\ref{prop:DeterminantalVarietyNoSerendipitousPoint} ensures that, for every sequence $(X_i)_{i \in \N}$ in $\R_{\le r}^{m \times n}$ that has at least one accumulation point that is B-stationary for~\eqref{eq:OptiDeterminantalVariety} and for every $\varepsilon \in (0, \infty)$, there exists $i \in \N$ such that $\s(X_i; f, \R_{\le r}^{m \times n}) \le \varepsilon$. This provides the following stopping criterion for every method that generates such a sequence: given $\varepsilon \in (0, \infty)$, stop the method at iteration
\begin{equation*}
i_\varepsilon \coloneq \min\{i \in \N \mid \s(X_i; f, \R_{\le r}^{m \times n}) \le \varepsilon\}.
\end{equation*}

\section{Sufficient-descent maps}
\label{sec:SufficientDescentMaps}
This section, based on \cite[\S 1.3]{Polak1971} and \cite[\S 7.6]{LuenbergerYe}, provides the theoretical framework used to analyze $\ppgdr$ in Section~\ref{sec:P2GDRdeterminantalVariety} and $\ppgd$--$\pgd$ in Section~\ref{sec:P2GD--PGD}.
Given a Euclidean vector space~$\mathcal{E}$, a nonempty closed subset $C$ of $\mathcal{E}$, and a function $\varphi : \mathcal{E} \to \R$ that is continuous on $C$, the framework considers the problem
\begin{equation}
\label{eq:MinContFunctionClosedSet}
\min_{x \in C} \varphi(x)
\end{equation}
of minimizing $\varphi$ on $C$ and the model of an iterative method presented in Algorithm~\ref{algo:ModelIterativeOptimizationMethod}.

\begin{algorithm}[H]
\caption{Model of an iterative method aiming at solving problem~\eqref{eq:MinContFunctionClosedSet}}
\label{algo:ModelIterativeOptimizationMethod}
\begin{algorithmic}[1]
\Require
$(\mathcal{E}, C, S, \varphi, A)$, where $\mathcal{E}$ is a Euclidean vector space, $C$ is a nonempty closed subset of~$\mathcal{E}$, $S$ is a nonempty proper subset of~$C$, $\varphi : \mathcal{E} \to \R$ is continuous on $C$, and $A$ is a set-valued map from $C$ to $C$ such that $A(x) \ne \emptyset$ for all $x \in C$.
\Input
$x_0 \in C$.
\Output
a sequence in $C$.

\State
$i \gets 0$;
\While
{$x_i \notin S$}
\State
Choose $x_{i+1} \in A(x_i)$;
\State
$i \gets i+1$;
\EndWhile
\end{algorithmic}
\end{algorithm}

If $\varphi$ is differentiable on $C$, a possible $S$ is the set of B-stationary points of~\eqref{eq:MinContFunctionClosedSet}, $S \coloneq \{x \in C \mid -\nabla \varphi(x) \in \regnorcone{C}{x}\}$. If Algorithm~\ref{algo:ModelIterativeOptimizationMethod} generates a finite sequence, then the last element of the sequence is in $S$. The rest of this section focuses on the nontrivial case where it generates an infinite sequence. By Proposition~\ref{prop:SufficientDescentMapAccumulation}, the accumulation points of every infinite sequence generated by Algorithm~\ref{algo:ModelIterativeOptimizationMethod} lie in $S$ if the set-valued map $A$, called the \emph{iteration map} of Algorithm~\ref{algo:ModelIterativeOptimizationMethod}, is a sufficient-descent map in the sense of Definition~\ref{def:SufficientDescentMap}. Moreover, by Proposition~\ref{prop:SufficientDescentMapFiniteUnion}, the finite union of sufficient-descent maps is a sufficient-descent map, an instrumental property to analyze hybrid methods such as $\ppgd$--$\pgd$.

\begin{definition}[sufficient-descent map]
\label{def:SufficientDescentMap}
A \emph{$(\varphi, C, S)$-sufficient-descent map} is a set-valued map $A$ from $C$ to $C$ that satisfies the following two conditions:
\begin{enumerate}
\item for all $x \in C$, $A(x) \ne \emptyset$;
\item for every $\ushort{x} \in C \setminus S$, there exist $\varepsilon(\ushort{x}), \delta(\ushort{x}) \in (0, \infty)$ such that, for all $x \in \ball[\ushort{x}, \varepsilon(\ushort{x})] \cap C$, it holds that $\varphi(x)-\sup\varphi(A(x)) \ge \delta(\ushort{x})$.
\end{enumerate}
If $\varphi$ is differentiable on $C$, a \emph{$(\varphi, C)$-sufficient-descent map} is a $(\varphi, C, S)$-sufficient-descent map with $S \coloneq \{x \in C \mid -\nabla \varphi(x) \in \regnorcone{C}{x}\}$.
\end{definition}

In other words, every $\ushort{x} \in C \setminus S$ has a neighborhood where the decrease in $\varphi$ produced by the map $A$ is bounded away from zero.
Examples of $(f, \R_{\le r}^{m \times n})$-sufficient-descent maps include the monotone $\pgd$ map (as can be deduced from \cite{OlikierWaldspurger}), the $\ppgdr$ map (Lemma~\ref{lemma:P2GDRmapPolak}), and the $\rfdr$ map \cite[Proposition~6.1]{OlikierAbsil2023}.

\begin{proposition}[accumulation points in a set]
\label{prop:SufficientDescentMapAccumulation}
Let $A$ be a $(\varphi, C, S)$-sufficient-descent map and $(x_i)_{i \in \N}$ be a sequence in $C \setminus S$ such that $x_{i+1} \in A(x_i)$ for all $i \in \N$. Then, all accumulation points of $(x_i)_{i \in \N}$ are in $S$.
\end{proposition}

\begin{proof}
Let $(x_{i_k})_{k \in \N}$ be a subsequence converging to $x \in C$. For the sake of contradiction, assume that $x \notin S$. Let $\varepsilon(x)$ and $\delta(x)$ be given by Definition~\ref{def:SufficientDescentMap}. There exists $K \in \N$ such that, for all integers $k \ge K$, $x_{i_k} \in \ball[x, \varepsilon(x)]$ and thus $\varphi(x_{i_k})-\varphi(x_{i_k+1}) \ge \delta(x)$. Thus, since $(\varphi(x_i))_{i \in \N}$ is decreasing, for all integers $k \ge K$, $\varphi(x_{i_k})-\varphi(x_{i_{k+1}}) \ge \delta(x)$. Since $\varphi$ is continuous, $(\varphi(x_{i_k}))_{k \in \N}$ converges to $\varphi(x)$. Therefore, letting $k$ tend to infinity in the last inequality yields a contradiction.
\end{proof}

Following \cite[Remark~14]{Polak1971}, which states that it is usually better to determine whether a method generates a sequence having at least one accumulation point by examining the method in the light of the specific problem to which one wishes to apply it, no condition ensuring the existence of a convergent subsequence is imposed in Proposition~\ref{prop:SufficientDescentMapAccumulation}. As a reminder, a sequence $(x_i)_{i \in \N}$ in $\mathcal{E}$ has at least one accumulation point if and only if $\liminf_{i \to \infty} \norm{x_i} < \infty$.

\begin{proposition}[finite union of sufficient-descent maps]
\label{prop:SufficientDescentMapFiniteUnion}
Let $s$ be a positive integer. For every $i \in \{1, \dots, s\}$, let $A_i$ be a $(\varphi, C, S)$-sufficient-descent map. Then, every map that associates a nonempty subset of $\bigcup_{i=1}^s A_i(x)$ with every $x \in C$ is a $(\varphi, C, S)$-sufficient-descent map.
\end{proposition}

\begin{proof}
Let $\ushort{x} \in C \setminus S$. For every $i \in \{1, \dots, s\}$, there exist $\varepsilon_i(\ushort{x}), \delta_i(\ushort{x}) \in (0, \infty)$ such that, for all $x \in \ball[\ushort{x}, \varepsilon_i(\ushort{x})] \cap C$, it holds that $\varphi(x)-\sup\varphi(A_i(x)) \ge \delta_i(\ushort{x})$. Define $\varepsilon(\ushort{x}) \coloneq \min_{i \in \{1, \dots, s\}} \varepsilon_i(\ushort{x})$ and $\delta(\ushort{x}) \coloneq \min_{i \in \{1, \dots, s\}} \delta_i(\ushort{x})$. Then, for all $x \in \ball[\ushort{x}, \varepsilon(\ushort{x})] \cap C$, we have $\varphi(x)-\sup\varphi(\bigcup_{i=1}^s A_i(x)) \ge \delta(\ushort{x})$.
\end{proof}

For example, Proposition~\ref{prop:SufficientDescentMapFiniteUnion} implies that every sequence $(X_i)_{i \in \N}$ in $\R_{\le r}^{m \times n}$ such that, for every $i \in \N$, $X_{i+1}$ is an image of $X_i$ under the $\rfdr$, monotone $\pgd$, or $\ppgdr$ map accumulates at B-stationary points of~\eqref{eq:OptiDeterminantalVariety}. This is exploited in Section~\ref{sec:P2GD--PGD} to design and analyze $\ppgd$--$\pgd$.

\section{An upper bound on the distance to the determinantal variety from a tangent line}
\label{sec:UpperBoundDistanceToDeterminantalVarietyFromTangent}
The analysis of $\ppgd$ in Section~\ref{sec:P2GDmapDeterminantalVariety} is based on Proposition~\ref{prop:GlobalSecondOrderUpperBoundDistanceToRealDeterminantalVarietyFromTangentLine}, which offers an upper bound on the distance to $\R_{\le r}^{m \times n}$ from each of its tangent lines.
A trivial upper bound on $\dist(X+Z, \R_{\le r}^{m \times n})$ holding for all $X \in \R_{\le r}^{m \times n}$ and $Z \in \tancone{\R_{\le r}^{m \times n}}{X}$ is $\norm{Z}$, and \cite[Proposition~3.6]{SchneiderUschmajew2015} tightens it to $\frac{1}{\sqrt{2}} \norm{Z}$. However, an upper bound proportional to $\norm{Z}^2$ is needed in the proof of Proposition~\ref{prop:P2GDmapUpperBoundCost}; such a bound is given and the proportionality factor is shown to be $1/\sigma_{\min}(X)$ up to a constant in Proposition~\ref{prop:GlobalSecondOrderUpperBoundDistanceToRealDeterminantalVarietyFromTangentLine}.

\begin{proposition}
\label{prop:GlobalSecondOrderUpperBoundDistanceToRealDeterminantalVarietyFromTangentLine}
For all $X \in \R_{\le r}^{m \times n} \setminus \{0_{m \times n}\}$,
\begin{equation*}
\frac{\sqrt{5}-1}{r+1} \frac{1}{2\sigma_{\min}(X)}
\le \sup_{Z \in \tancone{\R_{\le r}^{m \times n}}{X} \setminus \{0_{m \times n}\}} \frac{\dist(X+Z, \R_{\le r}^{m \times n})}{\norm{Z}^2}
\le \frac{1}{2\sigma_{\min}(X)}.
\end{equation*}
\end{proposition}

\begin{proof}
Let $\ushort{r} \coloneq \rank X$. We first establish the upper bound. Let
\begin{equation*}
X = [U \; U_\perp] \diag(\Sigma, 0_{m-\ushort{r} \times n-\ushort{r}}) [V \; V_\perp]^\tp
\end{equation*}
be an SVD, and $Z \in \tancone{\R_{\le r}^{m \times n}}{X} \setminus \{0_{m \times n}\}$. By~\eqref{eq:TanConeDeterminantalVariety}, there are $A_{1, 1} \in \R^{\ushort{r} \times \ushort{r}}$, $A_{1, 2} \in \R^{\ushort{r} \times n-\ushort{r}}$, $A_{2, 1} \in \R^{m-\ushort{r} \times \ushort{r}}$, and $A_{2, 2} \in \R_{\le r-\ushort{r}}^{m-\ushort{r} \times n-\ushort{r}}$ such that
\begin{equation*}
Z =
[U \; U_\perp]
\begin{bmatrix}
A_{1, 1} & A_{1, 2} \\ A_{2, 1} & A_{2, 2}
\end{bmatrix}
[V \; V_\perp]^\tp.
\end{equation*}
Define the function $\gamma : [0, \infty) \to \R_{\le r}^{m \times n}$ by
\begin{equation*}
\gamma(t) \coloneq \big(U+t(U_\perp A_{2, 1} + {\textstyle\frac{1}{2}}UA_{1, 1})\Sigma^{-1}\big) \Sigma \big(V+t(V_\perp A_{1, 2}^\tp + {\textstyle\frac{1}{2}}VA_{1, 1}^\tp)\Sigma^{-1}\big)^\tp + tU_\perp A_{2, 2} V_\perp^\tp,
\end{equation*}
where the first term is inspired from \cite[(13)]{ZhouEtAl2016}; $\gamma$ is well defined since the ranks of the two terms are respectively upper bounded by $\ushort{r}$ and $r-\ushort{r}$.
For all $t \in [0, \infty)$,
\begin{equation*}
\gamma(t)
= X + t Z + \frac{t^2}{4}
[U \; U_\perp]
\begin{bmatrix}
A_{1, 1} \\ 2A_{2, 1}
\end{bmatrix}
\Sigma^{-1}
\begin{bmatrix}
A_{1, 1} & 2A_{1, 2}
\end{bmatrix}
[V \; V_\perp]^\tp
\end{equation*}
thus
\begin{equation*}
\dist(X+tZ, \R_{\le r}^{m \times n})
\le \norm{(X+tZ)-\gamma(t)}
= \frac{t^2}{4} \left\|\begin{bmatrix} A_{1, 1}\Sigma^{-1}A_{1, 1} & 2A_{1, 1}\Sigma^{-1}A_{1, 2} \\ 2A_{2, 1}\Sigma^{-1}A_{1, 1} & 4A_{2, 1}\Sigma^{-1}A_{1, 2}\end{bmatrix}\right\|.
\end{equation*}
Observe that
\begin{align*}
&\left\|\begin{bmatrix} A_{1, 1}\Sigma^{-1}A_{1, 1} & 2A_{1, 1}\Sigma^{-1}A_{1, 2} \\ 2A_{2, 1}\Sigma^{-1}A_{1, 1} & 4A_{2, 1}\Sigma^{-1}A_{1, 2}\end{bmatrix}\right\|^2\\
&= \norm{A_{1, 1}\Sigma^{-1}A_{1, 1}}^2 + 4 \norm{A_{1, 1}\Sigma^{-1}A_{1, 2}}^2 + 4 \norm{A_{2, 1}\Sigma^{-1}A_{1, 1}}^2 + 16 \norm{A_{2, 1}\Sigma^{-1}A_{1, 2}}^2\\
&\le \norm{\Sigma^{-1}}_2^2 \left(\norm{A_{1, 1}}^4 + 4 \norm{A_{1, 1}}^2 \norm{A_{1, 2}}^2 + 4 \norm{A_{1, 1}}^2 \norm{A_{2, 1}}^2 + 16 \norm{A_{1, 2}}^2 \norm{A_{2, 1}}^2\right)\\
&\le \norm{\Sigma^{-1}}_2^2 \norm{Z}^4 \max_{\substack{x, y, z \in \R \\ x^2+y^2+z^2 \le 1}} x^4 + 4 x^2 y^2 + 4 x^2 z^2 + 16 y^2 z^2\\
&= 4 \norm{\Sigma^{-1}}_2^2 \norm{Z}^4\\
&= \frac{4}{\sigma_{\ushort{r}}^2(X)} \norm{Z}^4,
\end{align*}
where the first inequality follows from Proposition~\ref{prop:UpperBoundFrobeniusNormProduct}. Therefore, for all $t \in [0, \infty)$,
\begin{equation*}
\dist(X+tZ, \R_{\le r}^{m \times n}) \le t^2 \frac{1}{2 \sigma_{\ushort{r}}(X)} \norm{Z}^2.
\end{equation*}
Choosing $t = 1$ yields the upper bound.

We now establish the lower bound. Let
\begin{equation*}
X = [U \; U_\perp] \diag(\sigma_1, \dots, \sigma_{\ushort{r}}, 0_{m-\ushort{r} \times n-\ushort{r}}) [V \; V_\perp]^\tp
\end{equation*}
be an SVD, and observe that
\begin{equation*}
Z
\coloneq [U \; U_\perp] \sigma_{\ushort{r}} \diag(0_{\ushort{r}-1 \times \ushort{r}-1}, \left[\begin{smallmatrix} 0 & 1 \\ 1 & 0 \end{smallmatrix}\right], I_{r-\ushort{r}}, 0_{m-r-1 \times n-r-1}) [V \; V_\perp]^\tp
\in \tancone{\R_{\le r}^{m \times n}}{X}.
\end{equation*}
The nonzero singular values of
\begin{equation*}
X+Z = [U \; U_\perp] \diag(\sigma_1, \dots, \sigma_{\ushort{r}-1}, \sigma_{\ushort{r}} \left[\begin{smallmatrix} 1 & 1 \\ 1 & 0 \end{smallmatrix}\right], \sigma_{\ushort{r}} I_{r-\ushort{r}}, 0_{m-r-1 \times n-r-1}) [V \; V_\perp]^\tp
\end{equation*}
are $\sigma_1, \dots, \sigma_{\ushort{r}}$ and the absolute values of the eigenvalues of $\left[\begin{smallmatrix} 1 & 1 \\ 1 & 0 \end{smallmatrix}\right]$ multiplied by $\sigma_{\ushort{r}}$, i.e., $\frac{\sqrt{5}+1}{2}\sigma_{\ushort{r}}$ and $\frac{\sqrt{5}-1}{2}\sigma_{\ushort{r}}$. Thus, $\dist(X+Z, \R_{\le r}^{m \times n}) = \frac{\sqrt{5}-1}{2}\sigma_{\ushort{r}}$, $\norm{Z}^2 = (r-\ushort{r}+2)\sigma_{\ushort{r}}^2 \le (r+1)\sigma_{\ushort{r}}^2$, and the lower bound follows.
\end{proof}

Proposition~\ref{prop:GlobalSecondOrderUpperBoundDistanceToRealDeterminantalVarietyFromTangentLine} can be related to geometric principles. Let $\gamma$ be a curve on the embedded submanifold $\R_r^{m \times n}$ of $\R^{m \times n}$ \cite[Example~8.14]{Lee2003}. In view of the Gauss formula along a curve~\cite[Corollary~8.3]{Lee2018}, the normal part of the acceleration of $\gamma$ is given by $\mathrm{I\!I}(\gamma', \gamma')$, where $\mathrm{I\!I}$ denotes the second fundamental form. In view of~\cite[\S 4]{FepponLermusiaux2018}, the largest principal curvature of $\R_r^{m \times n}$ at $X$ is $1/\sigma_r(X)$; hence $\|\mathrm{I\!I}(\gamma'(t), \gamma'(t))\| \leq \|\gamma'(t)\|^2/\sigma_r(\gamma(t))$, and the bound is attained when $\gamma'(t)$ is along the corresponding principal direction.

\section{A lower bound on the step size of the $\ppgd$ map}
\label{sec:P2GDmapDeterminantalVariety}
The $\ppgd$ map, already briefly described in Section~\ref{subsec:StateOfTheArt}, is defined as Algorithm~\ref{algo:P2GDmap}. Corollary~\ref{coro:P2GDmapArmijoCondition}, based on Proposition~\ref{prop:P2GDmapUpperBoundCost}, provides an upper bound on the number of iterations performed in the backtracking procedure (while loop of Algorithm~\ref{algo:P2GDmap}). This serves as a basis for the convergence analysis conducted in Section~\ref{sec:P2GDRdeterminantalVariety} since the $\ppgd$ map is used as a subroutine by the $\ppgdr$ map (Algorithm~\ref{algo:P2GDRmap}).
Algorithm~\ref{algo:P2GDmap} corresponds to the iteration map of \cite[Algorithm~3]{SchneiderUschmajew2015} except that the initial step size for the backtracking procedure is chosen in a given bounded interval.

\begin{algorithm}[H]
\caption{$\ppgd$ map on $\R_{\le r}^{m \times n}$ (based on \cite[Algorithm~3]{SchneiderUschmajew2015})}
\label{algo:P2GDmap}
\begin{algorithmic}[1]
\Require
$(f, r, \ushort{\alpha}, \oshort{\alpha}, \beta, c)$, where $f : \R^{m \times n} \to \R$ is differentiable with $\nabla f$ locally Lipschitz continuous, $r < \min\{m, n\}$ is a positive integer, $0 < \ushort{\alpha} \le \oshort{\alpha} < \infty$, and $\beta, c \in (0, 1)$.
\Input
$X \in \R_{\le r}^{m \times n}$.
\Output
$Y \in \ppgd(X; f, r, \ushort{\alpha}, \oshort{\alpha}, \beta, c)$.

\State
\label{algo:P2GDmap:InitializationLineSearch}
Choose $G \in \proj{\tancone{\R_{\le r}^{m \times n}}{X}}{-\nabla f(X)}$, $\alpha \in [\ushort{\alpha}, \oshort{\alpha}]$, and $Y \in \proj{\R_{\le r}^{m \times n}}{X + \alpha G}$;
\While
{$f(Y) > f(X) - c \, \alpha \s(X; f, \R_{\le r}^{m \times n})^2$}
\State
$\alpha \gets \alpha \beta$;
\State
\label{algo:P2GDmap:LineSearch}
Choose $Y \in \proj{\R_{\le r}^{m \times n}}{X + \alpha G}$;
\EndWhile
\State
Return $Y$.
\end{algorithmic}
\end{algorithm}

By \cite[Lemma~1.2.3]{Nesterov2018}, since $\nabla f$ is locally Lipschitz continuous, i.e., for every closed ball $\mathcal{B} \subsetneq \R^{m \times n}$,
\begin{equation*}
\lip_{\mathcal{B}}(\nabla f) \coloneq \sup_{\substack{X, Y \in \mathcal{B} \\ X \ne Y}} \frac{\norm{\nabla f(X) - \nabla f(Y)}}{\norm{X-Y}} < \infty,
\end{equation*}
it holds for all $X, Y \in \mathcal{B}$ that
\begin{equation}
\label{eq:InequalityLipschitzContinuousGradient}
|f(Y) - f(X) - \ip{\nabla f(X)}{Y-X}| \le \frac{\lip_{\mathcal{B}}(\nabla f)}{2} \norm{Y-X}^2.
\end{equation}

\begin{proposition}
\label{prop:P2GDmapUpperBoundCost}
Let $X \in \R_{\le r}^{m \times n}$ and $\oshort{\alpha} \in (0, \infty)$. Let $\mathcal{B} \subsetneq \R^{m \times n}$ be a closed ball containing $\ball[X, 2 \oshort{\alpha} \s(X; f, \R_{\le r}^{m \times n})]$. Then, for all $G \in \proj{\tancone{\R_{\le r}^{m \times n}}{X}}{-\nabla f(X)}$ and $\alpha \in [0, \oshort{\alpha}]$, it holds that
\begin{equation}
\label{eq:P2GDmapUpperBoundCost}
\sup f(\proj{\R_{\le r}^{m \times n}}{X + \alpha G}) \le f(X) + \s(X; f, \R_{\le r}^{m \times n})^2 \alpha \left(-1+\kappa_\mathcal{B}(X; f, \oshort{\alpha})\alpha\right),
\end{equation}
where
\begin{equation*}
\kappa_\mathcal{B}(X; f, \oshort{\alpha}) \coloneq \left\{\begin{array}{ll}
\dfrac{1}{2} \lip\limits_\mathcal{B}(\nabla f) & \text{if } X = 0_{m \times n},\\
\dfrac{\norm{\nabla f(X)}}{2\sigma_{\min}(X)} + \dfrac{1}{2} \lip\limits_\mathcal{B}(\nabla f) \left(\oshort{\alpha}\dfrac{\s(X; f, \R_{\le r}^{m \times n})}{2\sigma_{\min}(X)}+1\right)^2 & \text{otherwise}.
\end{array}\right.
\end{equation*}
\end{proposition}

\begin{proof}
By Proposition~\ref{prop:TriangleInequalityProjection}, for all $G \in \proj{\tancone{\R_{\le r}^{m \times n}}{X}}{-\nabla f(X)}$ and $\alpha \in [0, \oshort{\alpha}]$, $\proj{\R_{\le r}^{m \times n}}{X + \alpha G} \subseteq \mathcal{B}$.
Let $G \in \proj{\tancone{\R_{\le r}^{m \times n}}{X}}{-\nabla f(X)}$. The proof of \eqref{eq:P2GDmapUpperBoundCost} is based on \eqref{eq:InequalityLipschitzContinuousGradient} and the equality $\ip{\nabla f(X)}{G} = -\s(X; f, \R_{\le r}^{m \times n})^2$, which holds by Proposition~\ref{prop:ProjectionOntoClosedCone} since $\tancone{\R_{\le r}^{m \times n}}{X}$ is a closed cone.
The result follows readily from \eqref{eq:InequalityLipschitzContinuousGradient} if $X = 0_{m \times n}$ since $\tancone{\R_{\le r}^{m \times n}}{0_{m \times n}} = \R_{\le r}^{m \times n}$. Let us therefore consider the case where $X \ne 0_{m \times n}$.
Let $L \coloneq \lip_\mathcal{B}(\nabla f)$. For all $\alpha \in [0, \oshort{\alpha}]$ and $Y \in \proj{\R_{\le r}^{m \times n}}{X + \alpha G}$,
\begin{align*}
f(Y)-f(X)
\le\:& \ip{\nabla f(X)}{Y-X} + \frac{L}{2} \norm{Y-X}^2\\
=\:& \ip{\nabla f(X)}{Y-(X+\alpha G)+\alpha G} + \frac{L}{2} \norm{Y-(X+\alpha G)+\alpha G}^2\\
=\:& - \alpha \s(X; f, \R_{\le r}^{m \times n})^2 + \ip{\nabla f(X)}{Y-(X+\alpha G)} + \frac{L}{2} \norm{Y-(X+\alpha G)+\alpha G}^2\\
\le\:& - \alpha \s(X; f, \R_{\le r}^{m \times n})^2 + \norm{\nabla f(X)}\dist(X+\alpha G, \R_{\le r}^{m \times n})\\
&+ \frac{L}{2} \left(\dist(X+\alpha G, \R_{\le r}^{m \times n})+\alpha \s(X; f, \R_{\le r}^{m \times n})\right)^2\\
\le\:& - \alpha \s(X; f, \R_{\le r}^{m \times n})^2 + \alpha^2 \frac{\norm{\nabla f(X)}}{2\sigma_{\min}(X)} \s(X; f, \R_{\le r}^{m \times n})^2\\
&+ \frac{L}{2} \left(\alpha^2 \frac{\s(X; f, \R_{\le r}^{m \times n})^2}{2\sigma_{\min}(X)} +\alpha \s(X; f, \R_{\le r}^{m \times n})\right)^2\\
=\:& \alpha \s(X; f, \R_{\le r}^{m \times n})^2 \left(-1+\alpha\left(\frac{\norm{\nabla f(X)}}{2\sigma_{\min}(X)} + \frac{L}{2} \left(\alpha\frac{\s(X; f, \R_{\le r}^{m \times n})}{2\sigma_{\min}(X)}+1\right)^2\right)\right)\\
\le\:& \alpha \s(X; f, \R_{\le r}^{m \times n})^2 \left(-1+\alpha\kappa_\mathcal{B}(X; f, \oshort{\alpha})\right),
\end{align*}
where the third inequality follows from Proposition~\ref{prop:GlobalSecondOrderUpperBoundDistanceToRealDeterminantalVarietyFromTangentLine}.
\end{proof}

\begin{corollary}
\label{coro:P2GDmapArmijoCondition}
Let $\mathcal{B}$ be a closed ball as in Proposition~\ref{prop:P2GDmapUpperBoundCost}. The while loop in Algorithm~\ref{algo:P2GDmap} terminates after at most
\begin{equation}
\label{eq:MaxNumIterationsP2GDmap}
\max\left\{0, \left\lceil\ln\bigg(\frac{1-c}{\alpha_0\kappa_\mathcal{B}(X; f, \oshort{\alpha})}\bigg)/\ln \beta\right\rceil\right\}
\end{equation}
iterations, where $\alpha_0$ is the initial step size chosen in line~\ref{algo:P2GDmap:InitializationLineSearch}, and every $Y \in \hyperref[algo:P2GDmap]{\ppgd}(X; f, r, \ushort{\alpha}, \oshort{\alpha}, \beta, c)$ satisfies the Armijo condition
\begin{equation}
\label{eq:P2GDmapArmijoCondition}
f(Y) \le f(X) - c \, \alpha \s(X; f, \R_{\le r}^{m \times n})^2
\end{equation}
with a step size $\alpha \in \big[\min\{\ushort{\alpha}, \frac{\beta(1-c)}{\kappa_\mathcal{B}(X; f, \oshort{\alpha})}\}, \oshort{\alpha}\big]$.
\end{corollary}

\begin{proof}
If $X$ is B-stationary for~\eqref{eq:OptiDeterminantalVariety}, then $G = 0_{m \times n}$, the while loop is not executed, and $Y = X$. Assume that $X$ is not B-stationary for~\eqref{eq:OptiDeterminantalVariety}. For all $\alpha \in (0, \infty)$,
\begin{equation*}
f(X) + \s(X; f, \R_{\le r}^{m \times n})^2 \alpha \big(-1+\kappa_\mathcal{B}(X; f, \oshort{\alpha})\alpha\big) \le f(X) - c \s(X; f, \R_{\le r}^{m \times n})^2 \alpha
\end{equation*}
if and only if
\begin{equation*}
\alpha \le \frac{1-c}{\kappa_\mathcal{B}(X; f, \oshort{\alpha})}.
\end{equation*}
Since the left-hand side of the first inequality is an upper bound on $f(\proj{\R_{\le r}^{m \times n}}{X + \alpha G})$ for all $\alpha \in (0, \oshort{\alpha}]$, the Armijo condition is satisfied if $\alpha \in (0, \min\{\oshort{\alpha}, \frac{1-c}{\kappa_\mathcal{B}(X; f, \oshort{\alpha})}\}]$.
Therefore, either the initial step size $\alpha_0$ chosen in $[\ushort{\alpha}, \oshort{\alpha}]$ satisfies the Armijo condition or the while loop ends after iteration $i \in \N \setminus \{0\}$ with $\alpha = \alpha_0 \beta^i$ such that $\frac{\alpha}{\beta} > \frac{1-c}{\kappa_\mathcal{B}(X; f, \oshort{\alpha})}$. In the second case, $i < 1+\ln(\frac{1-c}{\alpha_0\kappa_\mathcal{B}(X; f, \oshort{\alpha})})/\ln(\beta)$ and thus $i \le \lceil\ln(\frac{1-c}{\alpha_0\kappa_\mathcal{B}(X; f, \oshort{\alpha})})/\ln(\beta)\rceil$.
\end{proof}

\section{$\ppgdr$ and its convergence analysis}
\label{sec:P2GDRdeterminantalVariety}
The $\ppgdr$ map, already briefly described in Section~\ref{subsec:Contribution}, is defined as Algorithm~\ref{algo:P2GDRmap}.

\begin{algorithm}[H]
\caption{$\ppgdr$ map on $\R_{\le r}^{m \times n}$}
\label{algo:P2GDRmap}
\begin{algorithmic}[1]
\Require
$(f, r, \ushort{\alpha}, \oshort{\alpha}, \beta, c, \Delta)$, where $f : \R^{m \times n} \to \R$ is differentiable with $\nabla f$ locally Lipschitz continuous, $r < \min\{m, n\}$ is a positive integer, $0 < \ushort{\alpha} \le \oshort{\alpha} < \infty$, $\beta, c \in (0, 1)$, and $\Delta \in (0, \infty)$.
\Input
$X \in \R_{\le r}^{m \times n}$.
\Output
$Y \in \ppgdr(X; f, r, \ushort{\alpha}, \oshort{\alpha}, \beta, c, \Delta)$.

\For
{$i \in \{0, \dots, \rank X - \rank_\Delta X\}$}
\State
Choose $\hat{X}^i \in \proj{\R_{\rank X - i}^{m \times n}}{X}$;
\State
Choose $\tilde{X}^i \in \hyperref[algo:P2GDmap]{\ppgd}(\hat{X}^i; f, r, \ushort{\alpha}, \oshort{\alpha}, \beta, c)$;
\EndFor
\State
Return $Y \in \argmin_{\{\tilde{X}^i \mid i \in \{0, \dots, \rank X - \rank_\Delta X\}\}} f$.
\end{algorithmic}
\end{algorithm}

\begin{definition}
\label{def:P2GDR}
$\ppgdr$ is Algorithm~\ref{algo:ModelIterativeOptimizationMethod} applied to problem~\eqref{eq:OptiDeterminantalVariety} with $A$ the $\ppgdr$ map and $S$ the set of B-stationary points.
\end{definition}

$\ppgdr$ is analyzed in Section~\ref{subsec:ConvergenceAnalysis} based on the theoretical framework from Section~\ref{sec:SufficientDescentMaps}. The choice of the parameter $\Delta$ is briefly discussed in Section~\ref{subsec:ChoiceDelta}.

\subsection{Convergence analysis}
\label{subsec:ConvergenceAnalysis}
Lemma~\ref{lemma:P2GDRmapPolak} states that the $\ppgdr$ map is an $(f, \R_{\le r}^{m \times n})$-sufficient-descent map in the sense of Definition~\ref{def:SufficientDescentMap}: for every $\ushort{X} \in \R_{\le r}^{m \times n}$ that is not B-stationary for~\eqref{eq:OptiDeterminantalVariety}, the decrease in $f$ obtained by applying the $\ppgdr$ map to every $X \in \R_{\le r}^{m \times n}$ sufficiently close to $\ushort{X}$ is bounded away from zero.

\begin{lemma}
\label{lemma:P2GDRmapPolak}
The $\ppgdr$ map is an $(f, \R_{\le r}^{m \times n})$-sufficient-descent map in the sense of Definition~\ref{def:SufficientDescentMap}: for every $\ushort{X} \in \R_{\le r}^{m \times n}$ such that $\s(\ushort{X}; f, \R_{\le r}^{m \times n}) > 0$, there exist $\varepsilon(\ushort{X}), \delta(\ushort{X}) \in (0, \infty)$ such that, for all $X \in \ball[\ushort{X}, \varepsilon(\ushort{X})] \cap \R_{\le r}^{m \times n}$ and $Y \in \hyperref[algo:P2GDRmap]{\ppgdr}(X; f, r, \ushort{\alpha}, \oshort{\alpha}, \beta, c, \Delta)$,
\begin{equation}
\label{eq:PolakConditionIterationMap}
f(Y) - f(X) \le - \delta(\ushort{X}).
\end{equation}
\end{lemma}

\begin{proof}
Let $\ushort{X} \in \R_{\le r}^{m \times n}$ be such that $\s(\ushort{X}; f, \R_{\le r}^{m \times n}) > 0$. Define $\ushort{r} \coloneq \rank \ushort{X}$. This proof constructs $\varepsilon(\ushort{X})$ and $\delta(\ushort{X})$ based on the Armijo condition \eqref{eq:P2GDmapArmijoCondition} given in Corollary~\ref{coro:P2GDmapArmijoCondition}. The first step of the proof ends with the definition of $\delta(\ushort{X})$ in~\eqref{eq:P2GDRmapPolakDelta}. Then, the second and third steps define $\varepsilon(\ushort{X})$ for the cases where $\ushort{r} = r$ and $\ushort{r} < r$, respectively. The rank reduction mechanism of $\ppgdr$ plays a role only in the third step.

By Proposition~\ref{prop:ProjectionOntoClosedCone}, $\norm{\nabla f(\ushort{X})} \ge \s(\ushort{X}; f, \R_{\le r}^{m \times n})$. Since $\nabla f$ is continuous at $\ushort{X}$, there exists $\rho_1(\ushort{X}) \in (0, \infty)$ such that, for all $X \in \ball[\ushort{X}, \rho_1(\ushort{X})]$, $\norm{\nabla f(X)-\nabla f(\ushort{X})} \le \frac{1}{2} \norm{\nabla f(\ushort{X})}$ and hence, as $|\norm{\nabla f(X)}-\norm{\nabla f(\ushort{X})}| \le \norm{\nabla f(X)-\nabla f(\ushort{X})}$,
\begin{equation}
\label{eq:P2GDRmapPolakContinuityGradient}
\frac{1}{2} \norm{\nabla f(\ushort{X})} \le \norm{\nabla f(X)} \le \frac{3}{2} \norm{\nabla f(\ushort{X})}.
\end{equation}
Define
\begin{equation*}
\bar{\rho}(\ushort{X}) \coloneq 3 \oshort{\alpha} \norm{\nabla f(\ushort{X})} + \rho_1(\ushort{X}).
\end{equation*}
Then, for every $X \in \ball[\ushort{X}, \rho_1(\ushort{X})] \cap \R_{\le r}^{m \times n}$, the inclusion $\ball[X, 2 \oshort{\alpha} \s(X; f, \R_{\le r}^{m \times n})] \subseteq \ball[\ushort{X}, \bar{\rho}(\ushort{X})]$ holds since, for all $Z \in \ball[X, 2 \oshort{\alpha} \s(X; f, \R_{\le r}^{m \times n})]$,
\begin{equation*}
\norm{Z-\ushort{X}}
\le \norm{Z-X} + \norm{X-\ushort{X}}
\le 2 \oshort{\alpha} \s(X; f, \R_{\le r}^{m \times n}) + \rho_1(\ushort{X})
\le 2 \oshort{\alpha} \norm{\nabla f(X)} + \rho_1(\ushort{X})
\le \bar{\rho}(\ushort{X}),
\end{equation*}
where the third and fourth inequalities follow respectively from Proposition~\ref{prop:ProjectionOntoClosedCone} and the second inequality in~\eqref{eq:P2GDRmapPolakContinuityGradient}. Thus, for all $X \in \ball[\ushort{X}, \rho_1(\ushort{X})] \cap \R_{\le r}^{m \times n}$ and $\tilde{X} \in \hyperref[algo:P2GDmap]{\ppgd}(X; f, r, \ushort{\alpha}, \oshort{\alpha}, \beta, c)$, Corollary~\ref{coro:P2GDmapArmijoCondition} applies with the ball $\ball[\ushort{X}, \bar{\rho}(\ushort{X})]$, and \eqref{eq:P2GDmapArmijoCondition} yields
\begin{equation}
\label{eq:P2GDRmapPolakP2GDmapArmijoCondition}
f(\tilde{X}) \le f(X) - c \s(X; f, \R_{\le r}^{m \times n})^2 \min\left\{\ushort{\alpha}, \frac{\beta(1-c)}{\kappa_{\ball[\ushort{X}, \bar{\rho}(\ushort{X})]}(X; f, \oshort{\alpha})}\right\}.
\end{equation}
Define
\begin{equation}
\label{eq:P2GDRmapPolakKappa}
\oshort{\kappa}(\ushort{X}) \coloneq \left\{\begin{array}{ll}
\dfrac{1}{2} \lip_{\ball[\ushort{X}, \bar{\rho}(\ushort{X})]}(\nabla f) & \text{if } \ushort{X} = 0_{m \times n},\\
\dfrac{3\norm{\nabla f(\ushort{X})}}{2\sigma_{\ushort{r}}(\ushort{X})} + \dfrac{1}{2} \lip_{\ball[\ushort{X}, \bar{\rho}(\ushort{X})]}(\nabla f) \left(\oshort{\alpha}\dfrac{3\norm{\nabla f(\ushort{X})}}{2\sigma_{\ushort{r}}(\ushort{X})}+1\right)^2 & \text{if } \ushort{X} \ne 0_{m \times n}.
\end{array}\right.
\end{equation}
If $\ushort{X} \ne 0_{m \times n}$, then for all $X \in \ball[\ushort{X}, \frac{1}{2}\sigma_{\ushort{r}}(\ushort{X})]$ it holds that
\begin{equation*}
|\sigma_{\ushort{r}}(X)-\sigma_{\ushort{r}}(\ushort{X})|
\le \norm{X-\ushort{X}}
\le \frac{1}{2}\sigma_{\ushort{r}}(\ushort{X}),
\end{equation*}
where the first inequality follows from Proposition~\ref{prop:SingularValuesLipschitz}, and thus $\sigma_{\ushort{r}}(X) \ge \frac{1}{2}\sigma_{\ushort{r}}(\ushort{X})$. Hence, by Proposition~\ref{prop:ProjectionOntoClosedCone} and the second inequality in~\eqref{eq:P2GDRmapPolakContinuityGradient}, if $\ushort{X} \ne 0_{m \times n}$, then for all $X \in \ball[\ushort{X}, \min\{\rho_1(\ushort{X}), \frac{1}{2}\sigma_{\ushort{r}}(\ushort{X})\}] \cap \R_{\ushort{r}}^{m \times n}$ it holds that
\begin{equation}
\label{eq:P2GDRmapPolakKappaUpperBound}
\kappa_{\ball[\ushort{X}, \bar{\rho}(\ushort{X})]}(X; f, \oshort{\alpha}) \le \oshort{\kappa}(\ushort{X}).
\end{equation}
Therefore, by \eqref{eq:P2GDRmapPolakP2GDmapArmijoCondition}, for all $X \in \ball[\ushort{X}, \min\{\rho_1(\ushort{X}), \frac{1}{2}\sigma_{\ushort{r}}(\ushort{X})\}] \cap \R_{\ushort{r}}^{m \times n}$ and $\tilde{X} \in \hyperref[algo:P2GDmap]{\ppgd}(X; f, r, \ushort{\alpha}, \oshort{\alpha}, \beta, c)$,
\begin{equation}
\label{eq:P2GDRmapPolakP2GDmapArmijoCondition+}
f(\tilde{X}) \le f(X) - c \s(X; f, \R_{\le r}^{m \times n})^2 \min\left\{\ushort{\alpha}, \frac{\beta(1-c)}{\oshort{\kappa}(\ushort{X})}\right\}.
\end{equation}
The inequality \eqref{eq:P2GDRmapPolakKappaUpperBound} also holds if $\ushort{X} = X = 0_{m \times n}$, hence so does \eqref{eq:P2GDRmapPolakP2GDmapArmijoCondition+}.
Define
\begin{equation}
\label{eq:P2GDRmapPolakDelta}
\delta(\ushort{X}) \coloneq \left\{\begin{array}{ll}
\dfrac{c}{4} \s(\ushort{X}; f, \R_{\le r}^{m \times n})^2 \min\left\{\ushort{\alpha}, \dfrac{\beta(1-c)}{\oshort{\kappa}(\ushort{X})}\right\} & \text{if } \ushort{r} = r,\\
\dfrac{c \norm{\nabla f(\ushort{X})}^2}{12(\min\{m, n\}-r+1)} \min\left\{\ushort{\alpha}, \dfrac{\beta(1-c)}{\oshort{\kappa}(\ushort{X})}\right\} & \text{if } \ushort{r} < r.
\end{array}\right.
\end{equation}

Let us consider the case where $\ushort{r} = r$. On $\ball(\ushort{X}, \sigma_r(\ushort{X})) \cap \R_{\le r}^{m \times n} = \ball(\ushort{X}, \sigma_r(\ushort{X})) \cap \R_r^{m \times n}$, $\s(\cdot; f, \R_{\le r}^{m \times n})$ coincides with the norm of the Riemannian gradient of the restriction of $f$ to the smooth manifold $\R_r^{m \times n}$ (see Section~\ref{subsec:TangentNormalCones}), which is continuous \cite[\S 2.1]{SchneiderUschmajew2015}. Thus, there exists $\rho_2(\ushort{X}) \in (0, \sigma_r(\ushort{X}))$ such that $\s(\ball[\ushort{X}, \rho_2(\ushort{X})] \cap \R_r^{m \times n}; f, \R_{\le r}^{m \times n}) \subseteq [\frac{1}{2}\s(\ushort{X}; f, \R_{\le r}^{m \times n}), \frac{3}{2}\s(\ushort{X}; f, \R_{\le r}^{m \times n})]$. Define
\begin{equation*}
\varepsilon(\ushort{X}) \coloneq \min\left\{\rho_1(\ushort{X}), \rho_2(\ushort{X}), \frac{1}{2}\sigma_r(\ushort{X})\right\}.
\end{equation*}
Let $X \in \ball[\ushort{X}, \varepsilon(\ushort{X})] \cap \R_{\le r}^{m \times n}$ and $Y \in \hyperref[algo:P2GDRmap]{\ppgdr}(X; f, r, \ushort{\alpha}, \oshort{\alpha}, \beta, c, \Delta)$. There exists $\tilde{X} \in \hyperref[algo:P2GDmap]{\ppgd}(X; f, r, \ushort{\alpha}, \oshort{\alpha}, \beta, c)$ such that $f(Y) \le f(\tilde{X})$. Therefore, by \eqref{eq:P2GDRmapPolakP2GDmapArmijoCondition+} and \eqref{eq:P2GDRmapPolakDelta},
\begin{equation*}
f(Y)
\le f(\tilde{X})
\le f(X) - c \s(X; f, \R_{\le r}^{m \times n})^2 \min\left\{\ushort{\alpha}, \frac{\beta(1-c)}{\oshort{\kappa}(\ushort{X})}\right\}
\le f(X) - \delta(\ushort{X}).
\end{equation*}

Let us now consider the case where $\ushort{r} < r$. Since $f$ is continuous at $\ushort{X}$, there exists $\rho_0(\ushort{X}) \in (0, \infty)$ such that $f(\ball[\ushort{X}, \rho_0(\ushort{X})]) \subseteq [f(\ushort{X})-\delta(\ushort{X}), f(\ushort{X})+\delta(\ushort{X})]$.  Define
\begin{equation*}
\varepsilon(\ushort{X}) \coloneq \left\{\begin{array}{ll}
\min\{\Delta, \frac{1}{2}\rho_0(\ushort{X}), \frac{1}{2}\rho_1(\ushort{X})\} & \text{if } \ushort{X} = 0_{m \times n},\\[1mm]
\min\{\Delta, \frac{1}{2}\rho_0(\ushort{X}), \frac{1}{2}\rho_1(\ushort{X}), \frac{1}{4}\sigma_{\ushort{r}}(\ushort{X})\} & \text{if } \ushort{X} \ne 0_{m \times n}.
\end{array}\right. 
\end{equation*}
Let $X \in \ball[\ushort{X}, \varepsilon(\ushort{X})] \cap \R_{\le r}^{m \times n}$ and $Y \in \hyperref[algo:P2GDRmap]{\ppgdr}(X; f, r, \ushort{\alpha}, \oshort{\alpha}, \beta, c, \Delta)$. By Proposition~\ref{prop:LocalDeltaRank}, $\rank_\Delta X \le \ushort{r} \le \rank X$. Thus, $0 \le \rank X - \ushort{r} \le \rank X - \rank_\Delta X$ and there exist $\hat{X}^{\rank X-\ushort{r}} \in \proj{\R_{\ushort{r}}^{m \times n}}{X}$ and $\tilde{X}^{\rank X-\ushort{r}} \in \hyperref[algo:P2GDmap]{\ppgd}(\hat{X}^{\rank X-\ushort{r}}; f, r, \ushort{\alpha}, \oshort{\alpha}, \beta, c)$ such that $f(Y) \le f(\tilde{X}^{\rank X-\ushort{r}})$. By Proposition~\ref{prop:TriangleInequalityProjection}, $\hat{X}^{\rank X-\ushort{r}} \in \ball[\ushort{X},2\varepsilon(\ushort{X})]$. Therefore,
\begin{equation}
\label{eq:P2GDRmapPolakContinuityFunction}
f(\hat{X}^{\rank X-\ushort{r}}) \le f(X) + 2 \delta(\ushort{X}).
\end{equation}
Furthermore, Proposition~\ref{prop:ProjectionOntoClosedCone} and \eqref{eq:NormProjTanConeDeterminantalVariety} successively yield
\begin{equation*}
\norm{\nabla f(\hat{X}^{\rank X-\ushort{r}})}
\ge \s(\hat{X}^{\rank X-\ushort{r}}; f, \R_{\le r}^{m \times n})
\ge \sqrt{\frac{r-\ushort{r}}{\min\{m, n\}-\ushort{r}}} \norm{\nabla f(\hat{X}^{\rank X-\ushort{r}})}.
\end{equation*}
Thus, combining
\begin{equation*}
\frac{r-\ushort{r}}{\min\{m, n\}-\ushort{r}} \ge \frac{1}{\min\{m, n\}-r+1}
\end{equation*}
with \eqref{eq:P2GDRmapPolakContinuityGradient} yields
\begin{equation}
\label{eq:P2GDRmapPolakBoundsStationarityMeasure}
\s(\hat{X}^{\rank X-\ushort{r}}; f, \R_{\le r}^{m \times n}) \in \left[\frac{1}{2\sqrt{\min\{m, n\}-r+1}} \norm{\nabla f(\ushort{X})}, \frac{3}{2} \norm{\nabla f(\ushort{X})}\right].
\end{equation}
Hence,
\begin{align*}
f(Y)
&\le f(\tilde{X}^{\rank X-\ushort{r}})\\
&\le f(\hat{X}^{\rank X-\ushort{r}}) - c \s(\hat{X}^{\rank X-\ushort{r}}; f, \R_{\le r}^{m \times n})^2 \min\left\{\ushort{\alpha}, \frac{\beta(1-c)}{\oshort{\kappa}(\ushort{X})}\right\}\\
&\le f(X) + 2 \delta(\ushort{X}) - 3 \delta(\ushort{X})\\
&= f(X) - \delta(\ushort{X}),
\end{align*}
where the second inequality follows from \eqref{eq:P2GDRmapPolakP2GDmapArmijoCondition+} and the third from \eqref{eq:P2GDRmapPolakContinuityFunction}, \eqref{eq:P2GDRmapPolakBoundsStationarityMeasure}, and \eqref{eq:P2GDRmapPolakDelta}.
\end{proof}

\begin{theorem}
\label{thm:P2GDRPolak}
Consider a sequence generated by $\ppgdr$ (Definition~\ref{def:P2GDR}). If this sequence is finite, then its last element is B-stationary for~\eqref{eq:OptiDeterminantalVariety} in the sense of Definition~\ref{def:M/B-Stationarity}. If it is infinite, then all of its accumulation points are B-stationary for~\eqref{eq:OptiDeterminantalVariety}.
\end{theorem}

\begin{proof}
This follows from Proposition~\ref{prop:SufficientDescentMapAccumulation} and Lemma~\ref{lemma:P2GDRmapPolak}.
\end{proof}

\begin{corollary}
\label{coro:P2GDRPolak}
Let $(X_i)_{i \in \N}$ be a sequence generated by $\ppgdr$ (Definition~\ref{def:P2GDR}).
For every convergent subsequence $(X_{i_k})_{k \in \N}$, $\lim_{k \to \infty} \s(X_{i_k}; f, \R_{\le r}^{m \times n}) = 0$.
If $(X_i)_{i \in \N}$ is bounded, which is the case if the sublevel set $\{X \in \R_{\le r}^{m \times n} \mid f(X) \le f(X_0)\}$ is bounded, then $\lim_{i \to \infty} \s(X_i; f, \R_{\le r}^{m \times n}) = 0$ and $f$ has the same value at all accumulation points.
\end{corollary}

\begin{proof}
The two limits follow from Proposition~\ref{prop:DeterminantalVarietyNoSerendipitousPoint}. The final claim follows from the argument given in the proof of \cite[Theorem~65]{Polak1971}. Specifically, if $(X_i)_{i \in \N}$ is bounded, then it contains at least one convergent subsequence. Assume that $(X_{i_k})_{k \in \N}$ and $(X_{j_k})_{k \in \N}$ converge respectively to $\ushort{X}$ and $\oshort{X}$. The sequence $(f(X_i))_{i \in \N}$ is decreasing and, since $(X_i)_{i \in \N}$ is bounded and $f$ is continuous, it converges to $\inf_{i \in \N} f(X_i)$. Therefore, $f(\ushort{X}) = \lim_{k \to \infty} f(X_{i_k}) = \lim_{i \to \infty} f(X_i) = \lim_{k \to \infty} f(X_{j_k}) = f(\oshort{X})$.
\end{proof}

The analysis conducted in Sections~\ref{sec:P2GDmapDeterminantalVariety} and \ref{subsec:ConvergenceAnalysis} can be extended straightforwardly to the case where $f$ is only defined on an open subset of $\R^{m \times n}$ containing $\R_{\le r}^{m \times n}$.

\subsection{On the choice of $\Delta$}
\label{subsec:ChoiceDelta}
In many practical situations, when $\Delta$ is chosen reasonably small, all the iterates of $\ppgdr$ (Definition~\ref{def:P2GDR}) satisfy $\rank_\Delta X_i = \rank X_i$. In this case, the range of values in the for-loop of $\ppgdr$ always reduces to $\{0\}$, and $\ppgdr$ generates the same sequence $(X_i)_{i \in \N}$ as $\ppgd$. In this scenario, the only computational overhead in $\ppgdr$ is the computation of $\rank X_i$ and $\rank_\Delta X_i$. However, as mentioned in Section~\ref{subsec:Contribution}, for all $i \ge 1$, in view of line~\ref{algo:P2GDmap:LineSearch} of Algorithm~\ref{algo:P2GDmap}, it is reasonable to assume that $X_i$ has been obtained by a truncated SVD, in which case $\rank X_i$ and $\rank_\Delta X_i$ are immediately available, making the overhead insignificant. In summary, $\ppgdr$ offers stronger convergence properties than $\ppgd$, and while incurring an insignificant overhead in many practical situations.

However, the range of values in the for-loop of $\ppgdr$ can also be as large as $\{0, \dots, r\}$ (as recorded in Table~\ref{tab:ComparisonComputationalCostPerIteration}), and there are situations where this occurs each time the for-loop is reached (e.g., in the case of a bounded sublevel set, when $\Delta$ is chosen so large that $\rank_\Delta X = 0$ for all $X$ in the sublevel set). One can thus wonder whether it is possible to restrict (conditionally or not) the range of values in the for-loop while preserving the B-stationarity of the accumulation points. This is an open question. Nevertheless, in view of Proposition~\ref{prop:GlobalSecondOrderUpperBoundDistanceToRealDeterminantalVarietyFromTangentLine}, since in the neighborhood of every $X \in \R_{< r}^{m \times n}$, $\sigma_i$ can be arbitrarily small for every $i \in \{\rank X + 1, \dots, r\}$, it seems unlikely that $\ppgdr$ with a restricted for-loop can be analyzed along the lines of Section~\ref{subsec:ConvergenceAnalysis}. On the other hand, should the answer to the open question be negative, a counterexample other than the one from \cite[\S 2.2]{LevinKileelBoumal2023} would be required in view of~\cite[Remark~2.11]{LevinKileelBoumal2023}.

\section{A $\ppgd$--$\pgd$ hybrid}
\label{sec:P2GD--PGD}
This section takes advantage of Lemma~\ref{lemma:P2GDRmapPolak} and the theoretical framework from Section~\ref{sec:SufficientDescentMaps} to define a hybrid method, based on the $\ppgd$ map (Algorithm~\ref{algo:P2GDmap}) and the monotone $\pgd$ map (reviewed as Algorithm~\ref{algo:monotonePGDmap}), that accumulates at B-stationary points of~\eqref{eq:OptiDeterminantalVariety}. This hybrid method is called $\ppgd$--$\pgd$ and its iteration map, already briefly described in Section~\ref{subsec:Contribution}, is defined as Algorithm~\ref{algo:P2GD--PGDmap}. Other hybridizations are possible, but we emphasize this $\ppgd$--$\pgd$ hybrid because it combines two classic methods and does not rely on a rank reduction mechanism. It can be viewed as a way of speeding up the well-known $\pgd$ by offering opportunities to resort to the computationally cheaper $\ppgd$ (see Section~\ref{subsec:StateOfTheArt} and Table~\ref{tab:ComparisonComputationalCostPerIteration}) without relinquishing the property of accumulating at B-stationary points.

\begin{algorithm}[H]
\caption{monotone $\pgd$ map on $\R_{\le r}^{m \times n}$ (\cite[Algorithm~4.1]{OlikierWaldspurger} for problem~\eqref{eq:OptiDeterminantalVariety} with $\mu \coloneq f(x)$)}
\label{algo:monotonePGDmap}
\begin{algorithmic}[1]
\Require
$(f, r, \ushort{\alpha}, \oshort{\alpha}, \beta, c)$, where $f : \R^{m \times n} \to \R$ is differentiable with $\nabla f$ locally Lipschitz continuous, $r < \min\{m, n\}$ is a positive integer, $0 < \ushort{\alpha} \le \oshort{\alpha} < \infty$, and $\beta, c \in (0, 1)$.
\Input
$X \in \R_{\le r}^{m \times n}$.
\Output
$Y \in \pgd(X; f, r, \ushort{\alpha}, \oshort{\alpha}, \beta, c)$.

\State
Choose $\alpha \in [\ushort{\alpha}, \oshort{\alpha}]$ and $Y \in \proj{\R_{\le r}^{m \times n}}{X-\alpha\nabla f(X)}$;
\While
{$f(Y) > f(X) + c \ip{\nabla f(X)}{Y-X}$}
\State
$\alpha \gets \alpha \beta$;
\State
Choose $Y \in \proj{\R_{\le r}^{m \times n}}{X-\alpha\nabla f(X)}$;
\EndWhile
\State
Return $Y$.
\end{algorithmic}
\end{algorithm}

\begin{algorithm}[H]
\caption{$\ppgd$--$\pgd$ map on $\R_{\le r}^{m \times n}$}
\label{algo:P2GD--PGDmap}
\begin{algorithmic}[1]
\Require
$(f, r, \ushort{\alpha}, \oshort{\alpha}, \beta, c, \Delta)$, where $f : \R^{m \times n} \to \R$ is differentiable with $\nabla f$ locally Lipschitz continuous, $r < \min\{m, n\}$ is a positive integer, $0 < \ushort{\alpha} \le \oshort{\alpha} < \infty$, $\beta, c \in (0, 1)$, and $\Delta \in (0, \infty)$.
\Input
$X \in \R_{\le r}^{m \times n}$.
\Output
$Y \in \ppgd\text{--}\pgd(X; f, r, \ushort{\alpha}, \oshort{\alpha}, \beta, c, \Delta)$.

\If
{$\rank X = \rank_\Delta X$}
\State
\label{algo:P2GD--PGDmap:HybridStep}
Choose $Y \in \hyperref[algo:P2GDmap]{\ppgd}(X; f, r, \ushort{\alpha}, \oshort{\alpha}, \beta, c)$;
\Else
\State
Choose $Y \in \hyperref[algo:monotonePGDmap]{\pgd}(X; f, r, \ushort{\alpha}, \oshort{\alpha}, \beta, c)$;
\EndIf
\State
Return $Y$.
\end{algorithmic}
\end{algorithm}

\begin{lemma}
\label{lemma:P2GD--PGDmapSufficientDescentMap}
The $\ppgd$--$\pgd$ map is an $(f, \R_{\le r}^{m \times n})$-sufficient-descent map in the sense of Definition~\ref{def:SufficientDescentMap}.
\end{lemma}

\begin{proof}
If $\rank X = \rank_\Delta X$, then $\hyperref[algo:P2GDmap]{\ppgd}(X; f, r, \ushort{\alpha}, \oshort{\alpha}, \beta, c) = \hyperref[algo:P2GDRmap]{\ppgdr}(X; f, r, \ushort{\alpha}, \oshort{\alpha}, \beta, c, \Delta)$, hence Algorithm~\ref{algo:P2GD--PGDmap} remains unchanged if $\ppgd$ is replaced with $\ppgdr$ in line~\ref{algo:P2GD--PGDmap:HybridStep}. Since the monotone $\pgd$ map and the $\ppgdr$ map are $(f, \R_{\le r}^{m \times n})$-sufficient-descent maps (see \cite{OlikierWaldspurger} and Lemma~\ref{lemma:P2GDRmapPolak}), the $\ppgd$--$\pgd$ map is an $(f, \R_{\le r}^{m \times n})$-sufficient-descent map by Proposition~\ref{prop:SufficientDescentMapFiniteUnion}.
\end{proof}

Note that Lemma~\ref{lemma:P2GD--PGDmapSufficientDescentMap} remains true if ``$\cup~\hyperref[algo:monotonePGDmap]{\pgd}(X; f, r, \ushort{\alpha}, \oshort{\alpha}, \beta, c)$'' is inserted at the end of line~\ref{algo:P2GD--PGDmap:HybridStep} in Algorithm~\ref{algo:P2GD--PGDmap}. In other words, whenever line~\ref{algo:P2GD--PGDmap:HybridStep} of Algorithm~\ref{algo:P2GD--PGDmap} is reached, the monotone $\pgd$ map can be chosen instead of the $\ppgd$ map without affecting Lemma~\ref{lemma:P2GD--PGDmapSufficientDescentMap}, nor Theorem~\ref{thm:P2GD--PGD_Bstationary}.

\begin{definition}
\label{def:P2GD--PGD}
$\ppgd$--$\pgd$ is Algorithm~\ref{algo:ModelIterativeOptimizationMethod} applied to problem~\eqref{eq:OptiDeterminantalVariety} with $A$ the $\ppgd$--$\pgd$ map and $S$ the set of B-stationary points.
\end{definition}

\begin{theorem}
\label{thm:P2GD--PGD_Bstationary}
Consider a sequence generated by $\ppgd$--$\pgd$. If this sequence is finite, then its last element is B-stationary for~\eqref{eq:OptiDeterminantalVariety} in the sense of Definition~\ref{def:M/B-Stationarity}. If it is infinite, then all of its accumulation points are B-stationary for~\eqref{eq:OptiDeterminantalVariety}.
\end{theorem}

\begin{proof}
This follows from Proposition~\ref{prop:SufficientDescentMapAccumulation} and Lemma~\ref{lemma:P2GD--PGDmapSufficientDescentMap}.
\end{proof}

\section{Comparison of $\ppgdr$ and $\ppgd$--$\pgd$ with state-of-the-art methods}
\label{sec:ComparisonWithStateOfTheArt}
In this section, the proposed $\ppgdr$ and $\ppgd$--$\pgd$ are compared with the five state-of-the-art methods reviewed in Section~\ref{subsec:StateOfTheArt} and listed in Table~\ref{tab:OptimizationMethodsDeterminantalVariety}. The comparison is based on the computational cost per iteration, in Section~\ref{subsec:ComparisonComputationalCostPerIteration}, and empirically observed numerical performance on two weighted low-rank approximation problems, in Sections~\ref{subsec:ComparisonWLRAapo} and~\ref{subsec:ComparisonMatrixCompletion}. The first problem is defined in Section~\ref{subsec:P2GDandRFDapocalypsesWLRA}. The second problem is a matrix completion problem. The comparison focuses on the frequently encountered case where $3r \le \min\{m, n\}$.
As in Section~\ref{subsec:StateOfTheArt}, $\hrtr$ is studied with the rank factorization lift
\begin{equation}
\label{eq:RankFactorizationLift}
\varphi : \R^{m \times r} \times \R^{n \times r} \to \R^{m \times n} : (L, R) \mapsto LR^\tp
\end{equation}
from \cite[(1.1)]{LevinKileelBoumal2023}, the hook from \cite[Example~3.11]{LevinKileelBoumal2023}, and the Cauchy step at every iteration. This method minimizes the lifted cost function $g \coloneq f \circ \varphi$.

\subsection{Comparison based on the computational cost per iteration}
\label{subsec:ComparisonComputationalCostPerIteration}
In this section, the respective computational costs per iteration of the seven methods listed in Table~\ref{tab:OptimizationMethodsDeterminantalVariety} are compared based on detailed implementations of these methods involving only evaluations of $f$, $\nabla f$, and $\nabla^2 f$ and some operations from linear algebra:
\begin{enumerate}
\item mult.: matrix multiplication requiring at most $2mnr$ additions or multiplications of real numbers;
\item orth.: computation of an orthonormal basis of the column space of a matrix with $m$ or $n$ rows and at most $r$ columns;
\item small SVD: \emph{small-scale} truncated SVD, i.e., truncated SVD of rank at most $r$ of a matrix whose dimensions are at most $2r$;
\item medium SVD: \emph{medium-scale} compact SVD, i.e., compact SVD of a matrix with $m$ or $n$ rows and at most $r$ columns;
\item large SVD: \emph{large-scale} truncated SVD, i.e., truncated SVD of rank at most $\ushort{r} \in \{1, \dots, r\}$ of a matrix of size at least $m-r+1 \times n-r+1$ and at most $m \times n$;
\item eig.: computation of the smallest eigenvalue and an associated eigenvector of an order-$(m+n)r$ matrix.
\end{enumerate}
In this list, only the (truncated or compact) SVD and the computation of the smallest eigenvalue and an associated eigenvector rely on an iterative method rather than an algorithm with finite complexity such as those for QR factorization.

A part of the comparison was made in \cite[\S 7]{OlikierAbsil2023}, where the aforementioned detailed implementations of the $\ppgd$, $\ppgdr$, $\rfd$, and $\rfdr$ maps are provided. It remains to count matrix multiplications for those algorithms and to count all the above-listed operations for the $\pgd$, $\ppgd$--$\pgd$, and $\hrtr$ maps. Matrix multiplications that do not involve an $m$-by-$n$ matrix, either as a factor or as the product, are not counted. The comparison, summarized in Table~\ref{tab:ComparisonComputationalCostPerIteration}, focuses on the typical case where the input is not $0_{m \times n}$. The inequality $3r \le \min\{m, n\}$ implies that the projection in the right-hand side of~\eqref{eq:ProjTanConeDeterminantalVariety} and~\eqref{eq:ProjResTanConeDeterminantalVariety} requires a large-scale (truncated) SVD since the smaller dimension of the matrix to decompose is $\min\{m, n\}-\ushort{r}$, which is at least $\min\{m, n\}-r+1$ and thus greater than $2r$.
For an input of rank $\ushort{r}$, the counted matrix multiplications require $2mn\ushort{r}$ additions or multiplications of real numbers, if matrix multiplication is performed based on its definition \cite[Table~1.1.2]{GolubVanLoan}. In contrast, if $2r \le \sqrt{\min\{m, n\}}$ (which implies $3r < \min\{m, n\}$), then the uncounted matrix multiplications each require at most $mn$ additions or multiplications of real numbers since those that require the most additions or multiplications of real numbers are between $\max\{m, n\}$-by-$2r$ and $2r$-by-$r$ matrices, and  thus require $4\max\{m, n\}r^2$ additions or multiplications of real numbers.


The respective computational costs of the $\pgd$, $\ppgd$, and $\rfd$ maps depend on the number of iterations performed in the backtracking loop. By \cite[Corollary~6.2]{OlikierWaldspurger}, \eqref{eq:MaxNumIterationsP2GDmap}, and \cite[(11)]{OlikierAbsil2023}, given $X \in \R_{\le r}^{m \times n}$ as input and $\alpha \in (0, \infty)$ as initial step size, this number is upper bounded by
\begin{align}
\label{eq:MaxNumIterationsPGDmap}
&\max\left\{0, \left\lceil\ln\left(\frac{1-c}{\alpha\lip_{\ball[X, 2\alpha\norm{\nabla f(X)}]}(\nabla f)}\right)/\ln\beta\right\rceil\right\} && (\pgd),\\
\label{eq:MaxNumIterationsP2GDmapBis}
&\max\left\{0, \left\lceil\ln\left(\frac{1-c}{\alpha\kappa_{\ball[X, 2\alpha\s(X; f, \R_{\le r}^{m \times n})]}(X; f, \alpha)}\right)/\ln\beta\right\rceil\right\} && (\ppgd),\\
\label{eq:MaxNumIterationsRFDmap}
&\max\left\{0, \left\lceil\ln\left(\frac{2(1-c)}{\alpha\lip_{\ball[X, \alpha\s(X; f, \R_{\le r}^{m \times n})]}(\nabla f)}\right)/\ln\beta\right\rceil\right\} && (\rfd).
\end{align}
In view of the definition of the $\pgd$ map in \cite[Algorithm~4.1]{OlikierWaldspurger}, \eqref{eq:MaxNumIterationsPGDmap} justifies the $\pgd$ row in Table~\ref{tab:ComparisonComputationalCostPerIteration} for the general case where the gradient is not assumed to have a special structure such as sparse or low rank. Similarly, the $\ppgd$ and $\rfd$ rows in Table~\ref{tab:ComparisonComputationalCostPerIteration} follow from \eqref{eq:MaxNumIterationsP2GDmapBis} and \eqref{eq:MaxNumIterationsRFDmap}, as detailed in \cite[\S 7]{OlikierAbsil2023}.

It remains to justify the $\hrtr$ row. Proposition~\ref{prop:GradientHessianLiftedCostFunction} shows that, for every $(L, R) \in \R^{m \times r} \times \R^{n \times r}$, computing $\nabla g(L, R)$ and evaluating $\nabla^2 g(L, R)$ require respectively two and six matrix multiplications.
Moreover, according to the analysis conducted in Appendix~\ref{sec:PracticalImplementationHRTR}, every iteration of $\hrtr$ requires computing the smallest eigenvalue and an associated eigenvector of the order-$(m+n)r$ matrix representing $\nabla^2 g(L, R)$ in a given basis of $\R^{m \times r} \times \R^{n \times r}$, where $(L, R)$ is the current iterate.
Furthermore, every iteration that updates the current iterate requires a hook. The hook from \cite[Example~3.11]{LevinKileelBoumal2023} involves computing a matrix $U \in \st(\rank L, m)$ such that $\im U = \im L$ and a compact SVD of $R(L^\tp U)$, which is a medium-scale compact SVD.

\begin{table}[H]
\begin{center}
{\footnotesize
\begin{tabular}{l*{9}{c}}
\hline
\emph{Method} & $f$ & $\nabla f$ & $\nabla^2 f$ & mult. & orth. & small SVD & medium SVD & large SVD & eig.\\
\hline
\multirow{2}{*}{$\pgd$} & $1$ & \multirow{2}{*}{$1$} & \multirow{2}{*}{$0$} & \multirow{2}{*}{$0$} & \multirow{2}{*}{$0$} & \multirow{2}{*}{$0$} & \multirow{2}{*}{$0$} & $1$ & \multirow{2}{*}{$0$}\\
\cline{2-2}\cline{9-9}
& $1+\eqref{eq:MaxNumIterationsPGDmap}$ & & & & & & & $1+\eqref{eq:MaxNumIterationsPGDmap}$\\
\hline
\multirow{2}{*}{$\ppgd$} & $1$ & \multirow{2}{*}{$1$} & \multirow{2}{*}{$0$} & $2$ & $2$ & $1$ & \multirow{2}{*}{$0$} & $0$ & \multirow{2}{*}{$0$}\\
\cline{2-2}\cline{5-7}\cline{9-9}
& $1+\eqref{eq:MaxNumIterationsP2GDmapBis}$ & & & $4$ & $4$ & $1+\eqref{eq:MaxNumIterationsP2GDmapBis}$ & & $1$\\
\hline
\multirow{2}{*}{$\ppgdr$} & $1$ & $1$ & \multirow{2}{*}{$0$} & $2$ & $2$ & $1$ & \multirow{2}{*}{$0$} & $0$ & \multirow{2}{*}{$0$}\\
\cline{2-3}\cline{5-7}\cline{9-9}
& $1+\eqref{eq:MaxNumIterationsRFDmap} + r(1+\eqref{eq:MaxNumIterationsP2GDmapBis})$ & $r+1$ & & $4r-2$ & $4r-2$ & $r(1+\eqref{eq:MaxNumIterationsP2GDmapBis})$ & & $r$\\
\hline
\multirow{2}{*}{$\rfd$} & $1$ & \multirow{2}{*}{$1$} & \multirow{2}{*}{$0$} & $2$ & \multirow{2}{*}{$1$} & \multirow{2}{*}{$0$} & \multirow{2}{*}{$0$} & $0$ & \multirow{2}{*}{$0$}\\
\cline{2-2}\cline{5-5}\cline{9-9}
& $1+\eqref{eq:MaxNumIterationsRFDmap}$ & & & $4$ & & & & $1$\\
\hline
\multirow{2}{*}{$\rfdr$} & $1$ & $1$ & \multirow{2}{*}{$0$} & $2$ & \multirow{2}{*}{$0$} & \multirow{2}{*}{$0$} & $1$ & $0$ & \multirow{2}{*}{$0$}\\
\cline{2-3}\cline{5-5}\cline{8-9}
& $2(1+\eqref{eq:MaxNumIterationsRFDmap})$ & $2$ & & $6$ & & & $2$ & $1$\\
\hline
\multirow{2}{*}{$\hrtr$} & \multirow{2}{*}{$1$} & \multirow{2}{*}{$1$} & \multirow{2}{*}{$1$} & \multirow{2}{*}{$8$} & $0$ & \multirow{2}{*}{$0$} & $0$ & \multirow{2}{*}{$0$} & \multirow{2}{*}{$1$}\\
\cline{6-6}\cline{8-8}
& & & & & $1$ & & $1$ & \\
\hline
\end{tabular}}
\end{center}
\caption{Operations required by six methods aiming at solving problem~\eqref{eq:OptiDeterminantalVariety} with $r \le \min\{m, n\}/3$ to perform one iteration. The fields ``$f$'', ``$\nabla f$'', and ``$\nabla^2 f$'' respectively correspond to ``evaluation of $f$'', ``evaluation of $\nabla f$'', and ``evaluation of $\nabla^2 f$''. The other fields are defined in the beginning of Section~\ref{subsec:ComparisonComputationalCostPerIteration}; recall that only matrix multiplications involving an $m$-by-$n$ matrix are counted. When two subrows appear in a row, the upper entry corresponds to the best case and the lower one to the worst case. For the ``mult.'', ``orth.'', and ``large SVD'' columns of $\ppgd$ and $\rfd$, the best and worst cases are those where the rank of the input is equal to and smaller than $r$, respectively. $\ppgd$--$\pgd$ is not included in the table because its iteration map calls either the $\ppgd$ map or the monotone $\pgd$ map. An iteration of $\hrtr$ performs no ``orth.'' and no ``small SVD'' if and only if it does not update the iterate.}
\label{tab:ComparisonComputationalCostPerIteration}
\end{table}

\subsection{A family of weighted low-rank approximation problems}
\label{subsec:P2GDandRFDapocalypsesWLRA}
Let $\odot$ denote the entrywise multiplication and, for every $p \in (0, \infty)$, $\cdot^{\odot p}$ the entrywise $p$th power.
Given $A \in \R^{m \times n}$ and $W \in [0, \infty)^{m \times n} \setminus \{0_{m \times n}\}$, this section focuses on problem~\eqref{eq:OptiDeterminantalVariety} with $f : X \mapsto \frac{1}{2} \norm{W^{\odot\frac{1}{2}} \odot (X-A)}^2$, a weighted low-rank approximation (WLRA) problem. Specifically, it is proven that, for certain $A$ and $W$, $\ppgd$ and $\rfd$ can follow apocalypses (in the sense of \cite[Definition~2.7]{LevinKileelBoumal2023}, reviewed in Section~\ref{subsec:StationarityMeasuresOfStationarity}) when applied to that WLRA problem.

Assume that $3r \le \min\{m, n\}$. Let $r_1 \in \{1, \dots, r\}$ and $r_3 \in \{1, \dots, \min\{m, n\}-r\}$. Let $U \in \st(r+r_3, m)$ and $V \in \st(r+r_3, n)$. Let $U_1 \in \st(r_1, m)$, $U_2 \in \st(r-r_1, m)$, and $U_3 \in \st(r_3, m)$ such that $U = [U_1 \; U_2 \; U_3]$. Let $V_1 \in \st(r_1, n)$, $V_2 \in \st(r-r_1, n)$, and $V_3 \in \st(r_3, n)$ such that $V = [V_1 \; V_2 \; V_3]$. Let $S_1 \in \R_{r_1}^{r_1 \times r_1}$, $S_2, A_2 \in \R_{r-r_1}^{r-r_1 \times r-r_1}$, and $A_3 \in \R_{r_3}^{r_3 \times r_3}$. Define
\begin{align*}
X_0 &\coloneq [U_1 \; U_2 \; U_3] \diag(S_1, S_2, 0_{r_3 \times r_3}) [V_1 \; V_2 \; V_3]^\tp \in \R_r^{m \times n},\\
A &\coloneq [U_1 \; U_2 \; U_3] \diag(0_{r_1 \times r_1}, A_2, A_3) [V_1 \; V_2 \; V_3]^\tp \in \R_{r-r_1+r_3}^{m \times n}.
\end{align*}
If $r_1 \ge r_3$, then $\rank A \le r$ and $A$ is a global minimizer of the WLRA problem. For every $p, q \in \N$, let $1_{p \times q}$ denote the all-ones matrix in $\R^{p \times q}$.

\begin{proposition}
\label{prop:P2GDandRFDapocalypse(W)LRA}
Let $\alpha \in (0, 2) \setminus \{1\}$ and $\ushort{X} \coloneq U_2A_2V_2^\tp \in \R_{r-r_1}^{m \times n}$. Assume that either $W \in [0, 1]^{m \times n}$, $U = [I_{r+r_3} \; 0_{r+r_3 \times m-r-r_3}]^\tp$, $V = [I_{r+r_3} \; 0_{r+r_3 \times n-r-r_3}]^\tp$, and $W(r+1\mathord{:}r+r_3, r+1\mathord{:}r+r_3) \odot A_3 \ne 0_{r_3 \times r_3}$ or $W = 1_{m \times n}$. For all $i \in \N$, let
\begin{equation*}
X_i \coloneq (1_{m \times n}-\alpha W)^{\odot i} \odot X_0 + \left(1_{m \times n}-(1_{m \times n}-\alpha W)^{\odot i}\right) \odot \ushort{X} \in \R_{\le r}^{m \times n}.
\end{equation*}
Assume that, for all $i \in \N$, $X_i \in \R_r^{m \times n}$. Then:
\begin{itemize}
\item $(X_i)_{i \in \N}$ is the sequence generated by $\ppgd$ and $\rfd$ when applied to the WLRA problem with $\alpha$ as initial step size for the backtracking procedure, $\beta \in (0, 1)$, and $c \in (0, 1-\frac{\alpha}{2}]$;
\item $(X_i)_{i \in \N}$ converges to $\ushort{X}$;
\item $\lim_{i \to \infty} \s(X_i; f, \R_{\le r}^{m \times n}) = 0$;
\item $-\nabla f(\ushort{X}) = W \odot U_3A_3V_3^\tp \ne 0_{m \times n}$ hence, by \eqref{eq:NormProjTanConeDeterminantalVariety}, $\s(\ushort{X}; f, \R_{\le r}^{m \times n}) > 0$;
\item $\proj{\tancone{\R_{\le r}^{m \times n}}{\ushort{X}}}{-\nabla f(\ushort{X})} = \proj{\restancone{\R_{\le r}^{m \times n}}{\ushort{X}}}{-\nabla f(\ushort{X})} = \proj{\R_{\le r_1}^{m \times n}}{-\nabla f(\ushort{X})}$.
\end{itemize}
\end{proposition}

\begin{proof}
The second and fourth statements are clear. The fifth statement follows from~\eqref{eq:ProjTanConeDeterminantalVariety} and~\eqref{eq:ProjResTanConeDeterminantalVariety}. It remains to prove the first and third statements.

For all $X \in \R^{m \times n}$, $\nabla f(X) = W \odot (X-A)$. Let $W_1 \coloneq W(1\mathord{:}r_1, 1\mathord{:}r_1)$, $W_2 \coloneq W(r_1+1\mathord{:}r, r_1+1\mathord{:}r)$, and $W_3 \coloneq W(r+1\mathord{:}r+r_3, r+1\mathord{:}r+r_3)$. For all $i \in \N$,
\begin{equation*}
X_i = [U_1 \; U_2] \diag((1_{r_1 \times r_1}-\alpha W_1)^{\odot i} \odot S_1, (1_{r-r_1 \times r-r_1}-\alpha W_2)^{\odot i} \odot (S_2-A_2) + A_2) [V_1 \; V_2]^\tp
\end{equation*}
and
\begin{equation*}
\nabla f(X_i) = W \odot \left(U \diag((1_{r_1 \times r_1}-\alpha W_1)^{\odot i} \odot S_1, (1_{r-r_1 \times r-r_1}-\alpha W_2)^{\odot i} \odot (S_2-A_2), -A_3) V^\tp\right).
\end{equation*}
By \eqref{eq:ProjTanConeDeterminantalVariety} and \eqref{eq:ProjResTanConeDeterminantalVariety}, for all $i \in \N$,
\begin{align*}
\proj{\tancone{\R_{\le r}^{m \times n}}{X_i}}{\nabla f(X_i)}
&= \proj{\restancone{\R_{\le r}^{m \times n}}{X_i}}{\nabla f(X_i)}\\
&= W \odot (1_{m \times n}-\alpha W)^{\odot i} \odot \left([U_1 \; U_2] \diag(S_1, S_2-A_2) [V_1 \; V_2]^\tp\right)
\end{align*}
and thus
\begin{align*}
&\s(X_i; f, \R_{\le r}^{m \times n})^2\\
&= \norm{W_1 \odot (1_{r_1 \times r_1}-\alpha W_1)^{\odot i} \odot S_1}^2 + \norm{W_2 \odot (1_{r-r_1 \times r-r_1}-\alpha W_2)^{\odot i} \odot (S_2-A_2)}^2.
\end{align*}
The third statement is proven. Moreover, for all $i \in \N$,
\begin{align*}
X_{i+1} &= X_i - \alpha \proj{\tancone{\R_{\le r}^{m \times n}}{X_i}}{\nabla f(X_i)},\\
f(X_i)-f(X_{i+1})
&\ge \left(1-\frac{\alpha}{2}\right) \alpha \s(X_i; f, \R_{\le r}^{m \times n})^2 \ge c \alpha \s(X_i; f, \R_{\le r}^{m \times n})^2,
\end{align*}
which completes the proof.
\end{proof}

\subsection{Empirical comparison on a weighted low-rank approximation problem}
\label{subsec:ComparisonWLRAapo}
In this section, the six first-order methods from Table~\ref{tab:OptimizationMethodsDeterminantalVariety} are compared empirically on the WLRA problem from Section~\ref{subsec:P2GDandRFDapocalypsesWLRA}, with $r_1 = r_3$ to ensure that $A \in \R_r^{m \times n}$ is a global minimizer. The parameters of the problem are chosen as $m \coloneq 600$, $n \coloneq 400$, $r \coloneq 15$, and $r_1 \coloneq 10$, and those of the methods as $\ushort{\alpha} \coloneq \oshort{\alpha} \coloneq 0.8$, $\beta \coloneq 0.5$, $c \coloneq 0.1$, and $\Delta \in \{0.01, 0.1\}$. The initial iterate $X_0$ is the one defined in Section~\ref{subsec:P2GDandRFDapocalypsesWLRA}. The numerical experiments presented in this paper were performed in MATLAB R2025b on a laptop computer with an Apple M4 Pro chip featuring an 14-core CPU, a 20-core GPU, and 48 GB of memory.\footnote{The code is available at \url{https://github.com/golikier/LowRankOptimization}.} One hundred instances of the WLRA problem were generated randomly as follows:
\begin{itemize}
\item $A_2$ and $A_3$ are drawn from the standard normal distribution;
\item $W$ is drawn from the uniform distribution in the interval $(0, 1)$;
\item $S_1$ and $S_2$ are diagonal with monotonically nonincreasing elements drawn from the uniform distribution in the interval $(0, 1)$.
\end{itemize}

The number of instances on which the objective function was brought below $10^{-15}$ is represented as a function of time for each method in Figure~\ref{fig:ComparisonWLRAapo}. The number of instances on which the function was not brought below $10^{-15}$ within one minute is given in Table~\ref{tab:ComparisonWLRAapo}, which also gives the function value, measure of B-stationarity~\eqref{eq:NormProjectionNegativeGradientOntoTangentCone}, and $r$th singular value obtained after one minute on those instances.

\begin{figure}[H]
\begin{center}
\includegraphics[scale=0.93]{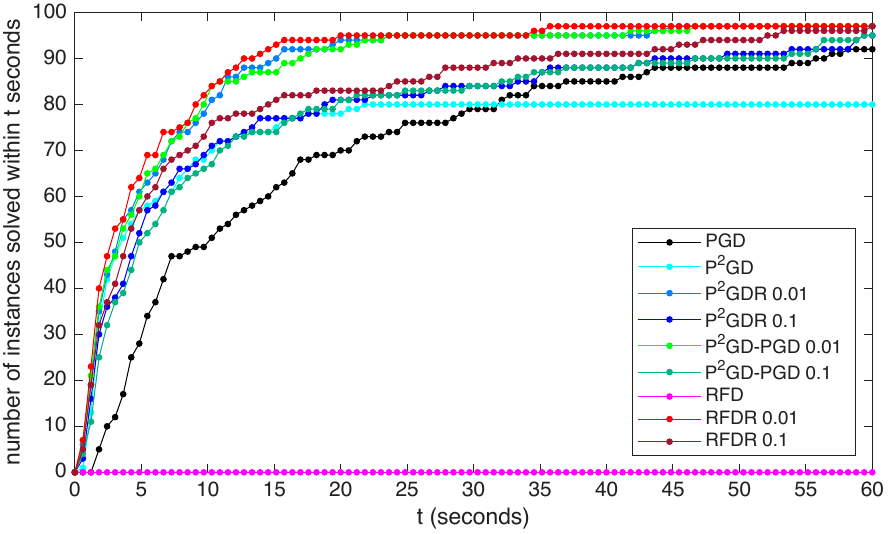}
\vspace*{-7mm}
\end{center}
\caption{Number of instances on which the objective function was brought below $10^{-15}$ as a function of time for the six first-order methods from Table~\ref{tab:OptimizationMethodsDeterminantalVariety} on $100$ randomly generated instance of the WLRA problem from Section~\ref{subsec:P2GDandRFDapocalypsesWLRA}. $\rfdr$ with $\Delta = 0.01$ performs best, followed by $\ppgdr$ and $\ppgd$--$\pgd$ both with $\Delta = 0.01$, which outperform $\ppgd$. $\ppgd$ is outperformed by all other methods except $\rfd$ after a bit more than $30$ seconds. $\rfd$ fails to bring the objective function below $10^{-15}$ on all instances. In this experiment, $\Delta = 0.01$ works better than $\Delta = 0.1$.}
\label{fig:ComparisonWLRAapo}
\end{figure}

In Table~\ref{tab:ComparisonWLRAapo}, the rows of $\ppgd$ and $\rfd$ contrast with those of the other methods: larger function value, smaller B-stationarity measure \eqref{eq:NormProjectionNegativeGradientOntoTangentCone}, and smaller $r$th singular value. This indicates that, on respectively twenty and all instances, $\ppgd$ and $\rfd$ follow the apocalypse described in Proposition~\ref{prop:P2GDandRFDapocalypse(W)LRA}. The objective function and B-stationarity measure \eqref{eq:NormProjectionNegativeGradientOntoTangentCone} are represented as functions of time for one of those twenty instances in Figure~\ref{fig:ComparisonWLRAapo70}. In contrast, the other methods converge to the global minimizer $A$ on every instance, even if they need more than one minute to bring the function below $10^{-15}$ on a few instances. For a given $\Delta$, the proposed $\ppgdr$ and $\ppgd$--$\pgd$ converge more slowly than $\rfdr$ but faster than $\pgd$, thereby confirming the computational advantage of relying on the $\ppgd$ map instead of the monotone $\pgd$ map.

\begin{table}[H]
\begin{center}
{\footnotesize
\begin{tabular}{*{8}{l}}
\hline
\emph{Method} & \emph{Failures} & \emph{Min. fun.} & \emph{Max. fun.} & \emph{Min. B} & \emph{Max. B} & \emph{Min. $\sigma_r$} & \emph{Max. $\sigma_r$}\\[1pt]
\hline
$\pgd$ & 8 & $1.2\cdot10^{-15}$ & $1.2\cdot10^{-6}$ & $1.2\cdot10^{-9}$ & $3.3\cdot10^{-6}$ & $6.5\cdot10^{-3}$ & $0.17$\\[1pt]
\hline
$\ppgd$ & 20 & $19.95$ & $36.09$ & $2.7\cdot10^{-14}$ & $5.9\cdot10^{-14}$ & $2.1\cdot10^{-93}$ & $1.4\cdot10^{-55}$\\[1pt]
\hline
$\ppgdr$ with $\Delta = 0.01$ & 3 & $1.2\cdot10^{-15}$ & $8.5\cdot10^{-7}$ & $7.3\cdot10^{-10}$ & $1.5\cdot10^{-6}$ & $6.5\cdot10^{-3}$ & $0.17$\\[1pt]
\hline
$\ppgdr$ with $\Delta = 0.1$ & 5 & $2.4\cdot10^{-15}$ & $7.3\cdot10^{-7}$ & $1.0\cdot10^{-9}$ & $1.3\cdot10^{-6}$ & $6.5\cdot10^{-3}$ & $0.17$\\[1pt]
\hline
$\ppgd$--$\pgd$ with $\Delta = 0.01$ & 3 & $2.5\cdot10^{-14}$ & $1.3\cdot10^{-6}$ & $3.3\cdot10^{-9}$ & $1.8\cdot10^{-6}$ & $6.5\cdot10^{-3}$ & $0.17$\\[1pt]
\hline
$\ppgd$--$\pgd$ with $\Delta = 0.1$ & 5 & $3.7\cdot10^{-15}$ & $1.2\cdot10^{-6}$ & $1.3\cdot10^{-9}$ & $1.7\cdot10^{-6}$ & $6.5\cdot10^{-3}$ & $0.17$\\[1pt]
\hline
$\rfd$ & 100 & $16.31$ & $36.09$ & $9.1\cdot10^{-16}$ & $5.1\cdot10^{-14}$ & $2.3\cdot10^{-64}$ & $1.9\cdot10^{-35}$\\[1pt]
\hline
$\rfdr$ with $\Delta = 0.01$ & 3 & $1.7\cdot10^{-15}$ & $8.9\cdot10^{-7}$ & $8.5\cdot10^{-10}$ & $1.5\cdot10^{-6}$ & $6.5\cdot10^{-3}$ & $0.17$\\[1pt]
\hline
$\rfdr$ with $\Delta = 0.1$ & 3 & $1.5\cdot10^{-15}$ & $4.1\cdot10^{-7}$ & $1.3\cdot10^{-9}$ & $1.0\cdot10^{-6}$ & $6.5\cdot10^{-3}$ & $0.057$\\[1pt]
\hline
\end{tabular}}
\vspace*{-4mm}
\end{center}
\caption{Performance of the six first-order methods from Table~\ref{tab:OptimizationMethodsDeterminantalVariety} on the instances of the WLRA problem from Section~\ref{subsec:P2GDandRFDapocalypsesWLRA} on which they failed to bring the objective function below $10^{-15}$ within one minute: number of those instances and function value, B-stationarity measure~\eqref{eq:NormProjectionNegativeGradientOntoTangentCone}, and $r$th singular value obtained after one minute on those instances. In contrast with the other methods, $\ppgd$ and $\rfd$ fail because they follow the apocalypse described in Proposition~\ref{prop:P2GDandRFDapocalypse(W)LRA}. Indeed, on the instances on which they fail: $\ppgd$ and $\rfd$ never bring the objective function below $16$, while the other methods always bring it below $1.3\cdot10^{-6}$; $\ppgd$ and $\rfd$ bring the B-stationarity measure~\eqref{eq:NormProjectionNegativeGradientOntoTangentCone} below $6\cdot10^{-14}$, while the other methods never bring it below $7.3\cdot10^{-10}$; $\ppgd$ and $\rfd$ bring the $r$th singular value below $1.9\cdot10^{-35}$, while the other methods never bring it below $6.4\cdot10^{-3}$.}
\label{tab:ComparisonWLRAapo}
\end{table}

\begin{figure}[H]
\begin{center}
\vspace*{-5mm}
\includegraphics[scale=1]{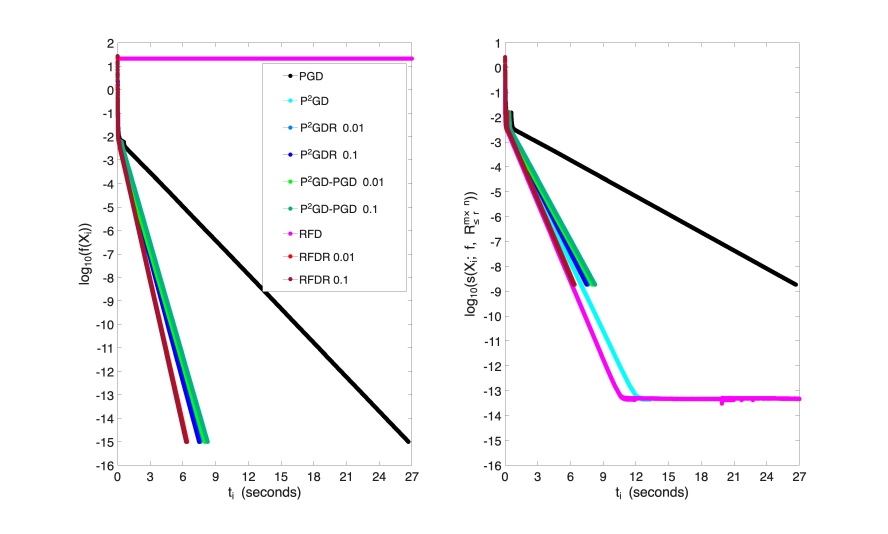}
\vspace*{-11mm}
\end{center}
\caption{Objective function and B-stationarity measure \eqref{eq:NormProjectionNegativeGradientOntoTangentCone} as functions of time for the six first-order methods from Table~\ref{tab:OptimizationMethodsDeterminantalVariety} on the 70th randomly generated instance of the WLRA problem from Section~\ref{subsec:P2GDandRFDapocalypsesWLRA}. Among the 17 instances on which both $\ppgd$ and $\rfd$ fail to bring the objective function below $10^{-15}$ within one minute but the other methods succeed, the ratio of the running time of $\ppgdr$, $\ppgd$--$\pgd$, or $\rfdr$ with a value of $\Delta$ to that with the other value of $\Delta$ is smaller than $1.5$ on 10 instances and greater than $2.8$ on the other 7 instances. Among those 10 instances, the 70th instance is the one on which each succeeding method reaches its maximal running time. On this instance, $\rfdr$ is the first method to reach the target value of $10^{-15}$ (within about $6.3$ seconds), followed by $\ppgdr$ (within about $7.5$ seconds), $\ppgd$--$\pgd$ (within about $8.1$ seconds), and $\pgd$ (within about $26.7$ seconds), while $\ppgd$ and $\rfd$ follow the apocalypse described in Proposition~\ref{prop:P2GDandRFDapocalypse(W)LRA} and never produce a value below $21$. All methods except $\ppgd$ and $\rfd$ performed between 11 091 and 11 125 iterations. With $\Delta = 0.01$, $\ppgdr$ and $\rfdr$ considered and used respectively $10$ and $11$ rank reductions, while $\ppgd$--$\pgd$ called the monotone $\pgd$ map only twice. With $\Delta = 0.1$, $\ppgdr$ and $\rfdr$ considered respectively $35$ and $33$ rank reductions and used respectively $12$ and $11$ of them, while $\ppgd$--$\pgd$ called the monotone $\pgd$ map $133$ times. Thus, except for a few iterations, the computational cost per iteration of these methods is the same as that of $\ppgd$ or $\rfd$. Recall that the B-stationarity measure \eqref{eq:NormProjectionNegativeGradientOntoTangentCone} is not lower semicontinuous.}
\label{fig:ComparisonWLRAapo70}
\end{figure}

\newpage

In exact arithmetic, $\ppgd$ and $\rfd$ each generate the sequence defined in Proposition~\ref{prop:P2GDandRFDapocalypse(W)LRA}, along which the $r$th singular value converges to $0$. However, Figures~\ref{fig:ComparisonWLRAapo}--\ref{fig:ComparisonWLRAapo52} indicate that, empirically, the two methods markedly differ in terms of success rate and speed. This can be explained as follows.
The implemented $\ppgd$ and $\rfd$ can be thought of as implementations of $\ppgdr$ and $\rfdr$, respectively, with $\Delta$ the smallest positive number that can be represented in binary64 ($\Delta = 2^{-1074} \approx 4.94\cdot10^{-324}$). Moreover, on $80$ instances, the $r$th singular value of the empirical iterates of $\ppgd$ became equal to $0$ within one minute---triggering the apocalypse-avoiding rank reduction mechanism of $\ppgdr$ albeit after a large number of iterations---whereas, on all instances, the $r$th singular value of every iterate generated by $\rfd$ remained greater than $10^{-74}$. This explains the different success rates. Furthermore, as long as their $r$th singular value is nonzero, the empirical iterates of $\ppgd$ and $\rfd$ are very close (Frobenius distance below $10^{-13}$ on the instances represented in Figures~\ref{fig:ComparisonWLRAapo70} and~\ref{fig:ComparisonWLRAapo52}) to the sequence defined in Proposition~\ref{prop:P2GDandRFDapocalypse(W)LRA}. The timings differ (as observable in the right-hand plots of Figures~\ref{fig:ComparisonWLRAapo70} and~\ref{fig:ComparisonWLRAapo52}) because $\ppgd$ has a higher computational cost per iteration than $\rfd$, as recorded in Table~\ref{tab:ComparisonComputationalCostPerIteration}.

\begin{figure}[H]
\begin{center}
\includegraphics[scale=1]{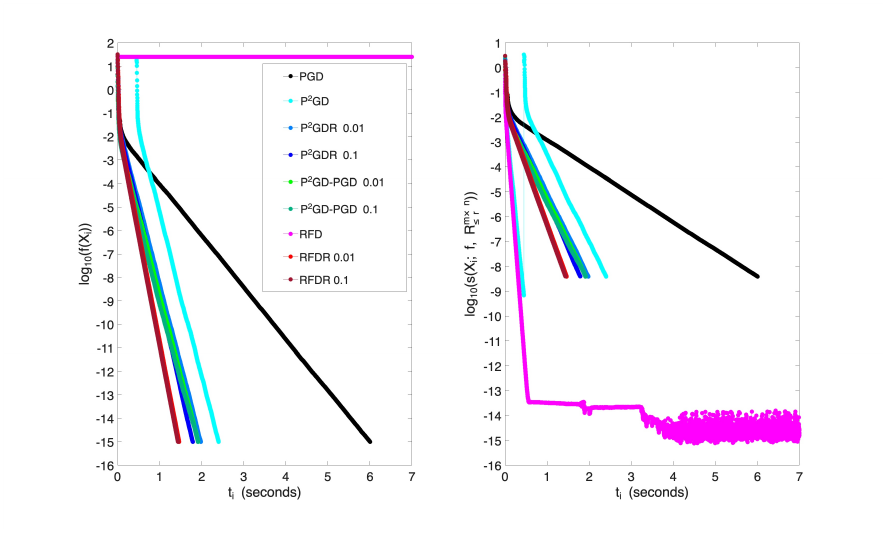}
\vspace*{-3mm}
\end{center}
\caption{Objective function and B-stationarity measure \eqref{eq:NormProjectionNegativeGradientOntoTangentCone} as functions of time for the six first-order methods from Table~\ref{tab:OptimizationMethodsDeterminantalVariety} on the 52nd randomly generated instance of the WLRA problem from Section~\ref{subsec:P2GDandRFDapocalypsesWLRA}. Among the 75 instances on which all methods except $\rfd$ bring the objective function below $10^{-15}$ within one minute, $\ppgd$ does so slower than $\ppgdr$, $\ppgd$--$\pgd$, and $\rfdr$ (with both values of $\Delta$) on 13 instances. Among those 13 instances, the 52nd instance is one where the time gap between $\ppgd$ and $\ppgdr$, $\ppgd$--$\pgd$, and $\rfdr$ is large and the influence of $\Delta$ on the running time is small. On this instance, $\rfdr$ is the first method to reach the target value of $10^{-15}$ (within about $1.5$ seconds), followed by $\ppgdr$ and $\ppgd$--$\pgd$ (within about $1.8$ seconds), $\ppgd$ (within about $2.4$ seconds), and $\pgd$ (within about $6$ seconds), while $\rfd$ follows the apocalypse described in Proposition~\ref{prop:P2GDandRFDapocalypse(W)LRA} and never produces a value below $25$. All methods except $\ppgd$ and $\rfd$ performed between 2438 and 2550 iterations. $\ppgd$ performed 3140 iterations. With $\Delta = 0.01$, $\ppgdr$ and $\rfdr$ considered and used respectively $1$ and $11$ rank reductions, while $\ppgd$--$\pgd$ called the monotone $\pgd$ map only twice. With $\Delta = 0.1$, $\ppgdr$ and $\rfdr$ considered $12$ rank reductions and used $11$ of them, while $\ppgd$--$\pgd$ called the monotone $\pgd$ map $5$ times. Thus, except for a few iterations, the computational cost per iteration of these methods is the same as that of $\ppgd$ or $\rfd$. The sequence generated by $\ppgd$ illustrates the lack of lower semicontinuity of the B-stationarity measure \eqref{eq:NormProjectionNegativeGradientOntoTangentCone}.}
\label{fig:ComparisonWLRAapo52}
\end{figure}

$\hrtr$ was not included in the comparison because, on the eleventh instance, for example, within $30$ minutes, it brought the objective function only at $6.06\cdot10^{-3}$. On this instance, $\pgd$ brought the objective function below $10^{-15}$ within $1.7$ seconds. Thus, on this instance, $\hrtr$ is at least $1000$ times slower than $\pgd$.

\subsection{Empirical comparison on a matrix completion problem}
\label{subsec:ComparisonMatrixCompletion}
In this section, the six first-order methods from Table~\ref{tab:OptimizationMethodsDeterminantalVariety} are compared empirically on a matrix completion problem similar to that presented in \cite[\S 3.4]{SchneiderUschmajew2015}, namely problem~\eqref{eq:OptiDeterminantalVariety} with $f : X \mapsto \frac{1}{2} \norm{W^{\odot\frac{1}{2}} \odot (X-A)}^2$, $A \in \R_r^{m \times n}$ to ensure that $A$ is a global minimizer, and $W \in \{0, 1\}^{m \times n} \setminus \{0_{m \times n}\}$. The parameters of the problem are chosen as $m \coloneq 450$, $n \coloneq 300$, and $r \coloneq 15$, and those of the methods as $\ushort{\alpha} \coloneq \oshort{\alpha} \coloneq 0.8$, $\beta \coloneq 0.5$, $c \coloneq 0.1$, and $\Delta \in \{0.01, 0.1\}$. The initial iterate $X_0$ is chosen in $\proj{\R_{\le r}^{m \times n}}{-\nabla f(0_{m \times n})}$. One hundred instances of the matrix completion problem were generated randomly as follows:
\begin{itemize}
\item $A \coloneq U_A \Sigma_A V_A^\tp$, where $\Sigma_A$ is an $r$-by-$r$ diagonal matrix whose entries are drawn from the uniform distribution in the interval $(0, 1)$ and $U_A \in \st(r, m)$ and $V_A \in \st(r, n)$ are the Q factors in the QR factorizations of an $m$-by-$r$ matrix and an $n$-by-$r$ matrix both drawn from the standard normal distribution;
\item $W$ contains $mn/20 < \dim \R_r^{m \times n} = (m+n-r)r$ ones at positions randomly selected, and its other entries are zero.
\end{itemize}

The running time and number of iterations needed by each method to bring the objective function below $10^{-15}$ on the $100$ instances are given in Table~\ref{tab:ComparisonTimeMatrixCompletion}. For the first instance, the objective function and the measure of B-stationarity \eqref{eq:NormProjectionNegativeGradientOntoTangentCone} are represented as functions of time in Figure~\ref{fig:PlotMatrixCompletion1}.

In this experiment, the six methods can be divided into three groups based on the running time. The fastest group, with a median running time close to $5$ seconds, contains $\ppgd$, $\ppgdr$, and $\ppgd$--$\pgd$. The second group, with a median running time close to $8$ seconds, contains $\rfd$ and $\rfdr$. The third group, with a median running time close to $11$ seconds, contains $\pgd$. The three groups are observable in Figure~\ref{fig:PlotMatrixCompletion1}.

\begin{figure}[H]
\begin{center}
\includegraphics[scale=1]{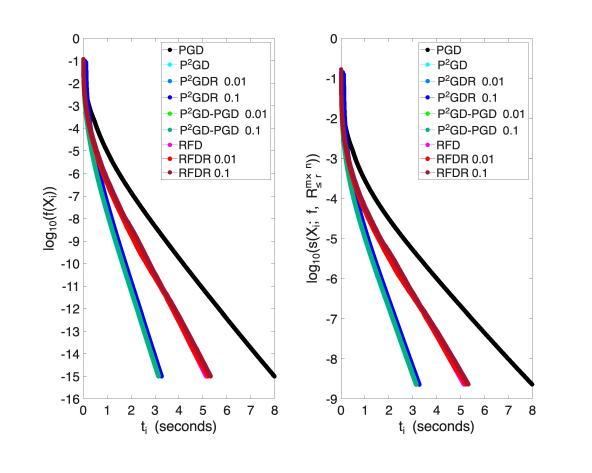}
\vspace*{-7mm}
\end{center}
\caption{Objective function and B-stationarity measure \eqref{eq:NormProjectionNegativeGradientOntoTangentCone} as functions of time for the six first-order methods from Table~\ref{tab:OptimizationMethodsDeterminantalVariety} on the first randomly generated instance of the matrix completion problem from Section~\ref{subsec:ComparisonMatrixCompletion}. For each method, the running time is smaller than the minimum running time in Table~\ref{tab:ComparisonTimeMatrixCompletion}. A plausible reason is that, in the experiment reported in Table~\ref{tab:ComparisonTimeMatrixCompletion}, it is verified at every iteration whether the objective function is smaller than $10^{-15}$. In contrast, this figure was obtained by performing the number of iterations needed to bring the objective function below $10^{-15}$; that number was computed in the experiment reported in Table~\ref{tab:ComparisonTimeMatrixCompletion}. On this instance, the $r$th singular value of all iterates generated by $\ppgd$ and $\rfd$ is greater than $0.01$. Thus, with $\Delta = 0.01$, $\ppgdr$ and $\ppgd$--$\pgd$ generated the same sequence as $\ppgd$, and $\rfdr$ generated the same sequence  as $\rfd$. Moreover, no rank reduction was considered by $\ppgdr$ or $\rfdr$. In contrast, with $\Delta = 0.1$, $\ppgdr$ and $\rfdr$ considered respectively $24$ and $2$ rank reductions but used none of them, thereby generating the same sequences as $\ppgd$ and $\rfd$, respectively. $\ppgd$--$\pgd$ used the monotone $\pgd$ map only twice.}
\label{fig:PlotMatrixCompletion1}
\end{figure}

\begin{table}[H]
\begin{center}
{\footnotesize
\begin{tabular}{*{7}{l}}
\hline
\emph{Method} & \emph{Min. time} & \emph{Median time} & \emph{Max. time} & \emph{Min. iter.} & \emph{Median iter.} & \emph{Max. iter.}\\[1pt]
\hline
$\pgd$ & 9.35 & 10.99 & 20.67 & 5192 & 6100 & 11 063\\[1pt]
\hline
$\ppgd$ & 3.94 & 5.15 & 9.66 & 4938 & 5929 & 11 153\\[1pt]
\hline
$\ppgdr$ with $\Delta = 0.01$ & 3.93 & 5.18 & 9.53 & 4938 & 5929 & 11 153\\[1pt]
\hline
$\ppgdr$ with $\Delta = 0.1$ & 4.24 & 5.38 & 9.71 & 4938 & 5929 & 11 153\\[1pt]
\hline
$\ppgd$--$\pgd$ with $\Delta = 0.01$ & 4.02 & 5.26 & 9.65 & 4938 & 5929 & 11 153\\[1pt]
\hline
$\ppgd$--$\pgd$ with $\Delta = 0.1$ & 4.01 & 5.18 & 8.89 & 5231 & 6040.5 & 11 309\\[1pt]
\hline
$\rfd$ & 5.99 & 7.99 & 12.96 & 9551 & 11 950.5 & 19 192\\[1pt]
\hline
$\rfdr$ with $\Delta = 0.01$ & 6.09 & 8.17 & 13.73 & 9551 & 11 950.5 & 19 192\\[1pt]
\hline
$\rfdr$ with $\Delta = 0.1$ & 6.14 & 8.21 & 13.72 & 9551 & 11 988 & 20 194\\[1pt]
\hline
\end{tabular}}
\vspace*{-2mm}
\end{center}
\caption{Running time (in seconds) and number of iterations needed by the six first-order methods from Table~\ref{tab:OptimizationMethodsDeterminantalVariety} to bring the objective function below $10^{-15}$ on the $100$ randomly generated instances of the matrix completion problem from Section~\ref{subsec:ComparisonMatrixCompletion}.}
\label{tab:ComparisonTimeMatrixCompletion}
\end{table}

$\hrtr$ was not included in the comparison because, on the first instance, for example, within $30$ minutes, it brought the objective function only at $2.14\cdot10^{-3}$. On this instance, $\pgd$ and $\ppgd$ brought the objective function below $10^{-15}$ within respectively $10$ and $5$ seconds. Thus, on this instance, $\hrtr$ is at least $180$ times slower than $\pgd$ and at least $360$ times slower than $\ppgd$.

\section{Conclusion}
\label{sec:Conclusion}
This paper proposes two first-order methods, called $\ppgdr$ and $\ppgd$--$\pgd$, that each generate a sequence in $\R_{\le r}^{m \times n}$ whose accumulation points are B-stationary for~\eqref{eq:OptiDeterminantalVariety} in the sense of Definition~\ref{def:M/B-Stationarity}. For problem~\eqref{eq:OptiDeterminantalVariety}, B-stationarity is the strongest necessary condition for local optimality (see Section~\ref{subsec:StationarityNotionsDeterminantalVariety}), and only a handful of methods in the literature provably enjoy this property (see Table~\ref{tab:OptimizationMethodsDeterminantalVariety}). A unique aspect of the proposed $\ppgdr$ is that it achieves this feat by only resorting to the classic and computationally cheap $\ppgd$ map combined with conditional rank reductions. Furthermore, the framework proposed in Section~\ref{sec:SufficientDescentMaps} makes it possible to produce hybrid methods that also accumulate at B-stationary points of~\eqref{eq:OptiDeterminantalVariety}. Among them, the $\ppgd$--$\pgd$ method introduced in Section~\ref{sec:P2GD--PGD} involves no rank reduction mechanism and combines two classic iteration maps, exploiting the lower computational cost of $\ppgd$ in comparison with $\pgd$ while preserving $\pgd$'s property of accumulating at B-stationary points.

The comparison conducted in Section~\ref{sec:ComparisonWithStateOfTheArt} shows that only $\rfdr$ can compete with the proposed $\ppgdr$ and $\ppgd$--$\pgd$. It has a lower worst-case computational cost per iteration, although the typical computational costs per iteration are comparable, and performs slightly better on the WLRA problem considered in Section~\ref{subsec:ComparisonWLRAapo}. However, $\ppgdr$ and $\ppgd$--$\pgd$ perform better than $\rfdr$ on the matrix completion problem studied in Section~\ref{subsec:ComparisonMatrixCompletion}. Moreover, due to their more streamlined design that does not require a restricted tangent cone, $\ppgdr$ and $\ppgd$--$\pgd$ are better candidates to extend to feasible sets where no such cone is known, such as the subset of $\R_{\le r}^{n \times n}$ containing the symmetric positive-semidefinite matrices, which appears notably in relaxations of combinatorial optimization problems; see, e.g., \cite{BurerMonteiro2003,BurerMonteiro2005,JourneeBachAbsilSepulchre2010,JiaEtAl} and references therein.

It is known (see Table~\ref{tab:OptimizationMethodsDeterminantalVariety}) that popular methods such as $\ppgd$ and $\rfd$ fail to accumulate at B-stationary points on certain instances of~\eqref{eq:OptiDeterminantalVariety}. The numerical experiments on a WLRA problem reported in Section~\ref{subsec:ComparisonWLRAapo} show that there exist problem instances and initial iterates such that $\ppgd$ and $\rfd$ are dramatically outperformed by the new methods. An important, yet loosely formulated, open question is how often such performance gaps arise in practice. In any case, methods such as the proposed $\ppgdr$ and $\ppgd$--$\pgd$ are of great interest because they provably accumulate at B-stationary points while often incurring small overhead compared with $\ppgd$ in practice, as observed in Section~\ref{subsec:ComparisonMatrixCompletion}.

\appendix

\section{Practical implementation of $\hrtr$}
\label{sec:PracticalImplementationHRTR}
This section details the implementation of $\hrtr$ with the rank factorization lift \eqref{eq:RankFactorizationLift} from \cite[(1.1)]{LevinKileelBoumal2023}. The lifted cost function is $g \coloneq f \circ \varphi$.
It is assumed that $\nabla f : \R^{m \times n} \to \R^{m \times n}$ is continuously differentiable, and $\nabla^2 f$ denotes the Hessian of~$f$, defined as the derivative of $\nabla f$ \cite[\S 5.5]{AbsilMahonySepulchre}, i.e., $\nabla^2 f : \R^{m \times n} \to \mathcal{L}(\R^{m \times n}) : X \mapsto (\nabla f)'(X)$, where $\mathcal{L}(\R^{m \times n})$ denotes the Banach space of all (continuous) linear operators on $\R^{m \times n}$.

\begin{proposition}
\label{prop:GradientHessianLiftedCostFunction}
For all $(L, R), (\dot{L}, \dot{R}) \in \R^{m \times r} \times \R^{n \times r}$,
\begin{align*}
\nabla g(L, R)
=~& (\nabla f(LR^\tp) R, \nabla f(LR^\tp)^\tp L),\\
\nabla^2 g(L, R)(\dot{L}, \dot{R})
=~& (\nabla^2 f(LR^\tp)(\dot{L}R^\tp+L\dot{R}^\tp)R+\nabla f(LR^\tp)\dot{R},\\
&\hphantom{(}\nabla^2 f(LR^\tp)(\dot{L}R^\tp+L\dot{R}^\tp)^\tp L+\nabla f(LR^\tp)^\tp\dot{L}).
\end{align*}
\end{proposition}

\begin{proof}
It holds that
\begin{equation*}
\varphi'(L, R)(\dot{L}, \dot{R}) = \dot{L}R^\tp+L\dot{R}^\tp.
\end{equation*}
By the chain rule,
\begin{equation*}
g'(L, R) = f'(\varphi(L, R)) \circ \varphi'(L, R).
\end{equation*}
Thus,
\begin{align*}
g'(L, R)(\dot{L}, \dot{R})
&= f'(\varphi(L, R))(\varphi'(L, R)(\dot{L}, \dot{R}))\\
&= \ip{\nabla f(\varphi(L, R))}{\varphi'(L, R)(\dot{L}, \dot{R})}\\
&= \ip{\nabla f(\varphi(L, R))}{\dot{L}R^\tp} + \ip{\nabla f(\varphi(L, R))}{L\dot{R}^\tp}\\
&= \ip{\nabla f(\varphi(L, R))R}{\dot{L}} + \ip{\nabla f(\varphi(L, R))^\tp L}{\dot{R}}\\
&= \ip{(((\nabla f) \circ \varphi)(L, R)R, ((\nabla f) \circ \varphi)(L, R)^\tp L)}{(\dot{L}, \dot{R})}
\end{align*}
and hence
\begin{equation*}
\nabla g(L, R) = (((\nabla f) \circ \varphi)(L, R)R, ((\nabla f) \circ \varphi)(L, R)^\tp L).
\end{equation*}
The formula for $\nabla^2 g(L, R)(\dot{L}, \dot{R})$ follows from the product rule and the chain rule.
\end{proof}

Given $(L, R) \in \R^{m \times r} \times \R^{n \times r}$, the eigenvalues of $\nabla^2 g(L, R)$, which is a self-adjoint linear operator on $\R^{m \times r} \times \R^{n \times r}$, are the eigenvalues of the matrix representing it in any basis of $\R^{m \times r} \times \R^{n \times r}$. Here, we use the orthonormal basis formed by the concatenation of the sequences $((\delta_{p,l})_{i,k=1}^{m,r}, 0_{n \times r})_{p,l=1}^{m,r}$ and $(0_{m \times r}, (\delta_{q,l})_{j,k=1}^{n,r})_{q,l=1}^{n,r}$. The matrix $H(L, R)$ representing $\nabla^2 g(L, R)$ in that basis can be formed as follows. For every $(i,k) \in \{1, \dots, m\} \times \{1, \dots, r\}$, the $(i-1)r+k$th column of $H(L, R)$ is the vector formed by concatenating the rows of $\nabla^2 f(LR^\tp)((\delta_{i,k})_{p,l=1}^{m,r}R^\tp)R$ and those of $\nabla^2 f(LR^\tp)((\delta_{i,k})_{p,l=1}^{m,r}R^\tp)^\tp L+\nabla f(LR^\tp)^\tp(\delta_{i,k})_{p,l=1}^{m,r}$. Then, for every $(j,k) \in \{1, \dots, n\} \times \{1, \dots, r\}$, the $mr+(j-1)r+k$th column of $H(L, R)$ is the vector formed by concatenating the rows of $\nabla^2 f(LR^\tp)(L{(\delta_{j,k})_{q,l=1}^{n,r}}^\tp)R+\nabla f(LR^\tp)(\delta_{j,k})_{q,l=1}^{n,r}$ and those of $\nabla^2 f(LR^\tp)(L{(\delta_{j,k})_{q,l=1}^{n,r}}^\tp)^\tp L$.

\section*{Acknowledgment}
The authors thank two anonymous reviewers for several helpful comments that improved the quality of the paper.

\bibliographystyle{plainurl}
\bibliography{golikier_bib_abbrv}
\end{document}